\documentclass{article}

\usepackage{microtype}
\usepackage{graphicx}
\usepackage{booktabs} 

\usepackage{hyperref}

\usepackage[accepted]{icml2024}

\usepackage{hyperref}       
\usepackage{url}            
\usepackage{booktabs}       
\usepackage{amsfonts}       
\usepackage{nicefrac}       
\usepackage{microtype}      
\usepackage{graphicx}
\usepackage{amsmath}
\usepackage{subcaption}
\usepackage{booktabs}
\usepackage{amsfonts}
\usepackage{algorithm}
\usepackage{algorithmic}
\usepackage{multirow}
\usepackage{xcolor}
\usepackage{amssymb}
\usepackage{mathtools}
\usepackage{amsthm}
\usepackage{colortbl}  
\usepackage{xcolor}

\usepackage[capitalize,noabbrev]{cleveref}

\theoremstyle{plain}
\newtheorem{theorem}{Theorem}[section]
\newtheorem{proposition}[theorem]{Proposition}
\newtheorem{lemma}[theorem]{Lemma}

\theoremstyle{definition}
\newtheorem{definition}[theorem]{Definition}
\newtheorem{assumption}[theorem]{Assumption}
\theoremstyle{remark}
\newtheorem{remark}[theorem]{Remark}

\usepackage[textsize=tiny]{todonotes}

\icmltitlerunning{Double Momentum Method for Lower-Level Constrained Bilevel Optimization}

\begin{document}

\twocolumn[
\icmltitle{Double Momentum Method for Lower-Level Constrained Bilevel Optimization}



\icmlsetsymbol{equal}{*}

\begin{icmlauthorlist}
\icmlauthor{Wanli Shi}{comp,yyy}
\icmlauthor{Yi Chang}{comp}
\icmlauthor{Bin Gu}{comp,yyy}
\end{icmlauthorlist}

\icmlaffiliation{yyy}{Mohamed bin Zayed University of Artificial Intelligence, UAE}
\icmlaffiliation{comp}{School of Artificial Intelligence, Jilin University, China}

\icmlcorrespondingauthor{Bin Gu}{jsgubin@gmail.com}
\icmlkeywords{Machine Learning, ICML}

\vskip 0.3in
]



\printAffiliationsAndNotice{\icmlEqualContribution} 

\begin{abstract}
Bilevel optimization (BO) has recently gained prominence in many machine learning applications due to its ability to capture the nested structure inherent in these problems. 
Recently, many hypergradient methods have been proposed as effective solutions for solving large-scale problems.
However, current hypergradient methods for the lower-level constrained bilevel optimization (LCBO) problems need very restrictive assumptions, namely, where optimality conditions satisfy the differentiability and invertibility conditions and lack a solid analysis of the convergence rate. What's worse, existing methods require either double-loop updates, which are sometimes
less efficient.
To solve this problem, in this paper, we propose a new hypergradient of LCBO leveraging the theory of nonsmooth implicit function theorem instead of using the restrive assumptions. In addition, we propose a \textit{single-loop single-timescale} algorithm based on the double-momentum method and adaptive step size method and prove it can return a $(\delta, \epsilon)$-stationary point with $\tilde{\mathcal{O}}(d_2^2\epsilon^{-4})$ iterations. Experiments on two applications demonstrate the effectiveness of our proposed method.
\end{abstract}

\section{Introduction}

Bilevel optimization (BO) \cite{bard2013practical,colson2007overview} plays a central role in various significant machine learning applications, including hyper-parameter optimization \cite{pedregosa2016hyperparameter,bergstra2011algorithms,bertsekas1976penalty,Shi_Gu_2021}, meta-learning \cite{feurer2015initializing,franceschi2018bilevel,rajeswaran2019meta}, reinforcement learning \cite{hong2020two,konda2000actor}. Generally speaking, the BO can be formulated as follows, 
\begin{align}
	&\min_{x\in\mathcal{X}} \ F(x) = f(x,y^*(x)) \\
	 &s.t. \ y^*(x)=\mathop{\arg\min}_{y\in\mathcal{Y}}g(x,y),\nonumber
\end{align}
where $\mathcal{X}$ and $\mathcal{Y}$ are convex subsets in $\mathbb{R}^{d_1}$ and $\mathbb{R}^{d_2}$, respectively.  It involves a competition between two parties or two objectives, and if one party makes its choice first, it will affect the optimal choice of the other party. 

Recently, hypergradient methods have shown great effectiveness in solving various large-scale bilevel optimization problems, where there is no constraint in the lower-level objective, \textit{i.e.}, $\mathcal{Y}=\mathbb{R}^{d_2}$. Specifically, \cite{franceschi2017forward,pedregosa2016hyperparameter,ji2021bilevel} proposed several double-loop algorithms to solve the BO problems. They first apply the gradient methods to approximate the solution to the lower-level problem and then implicit differentiable methods \cite{pedregosa2016hyperparameter,ji2021bilevel} or explicit differentiable methods \cite{franceschi2017forward} can be used to approximate the gradient of the upper-level objective w.r.t $x$, namely hypergradient, to update $x$. However, in some real-world applications, such as in a sequential game, the problems must be updated at the same time \cite{hong2020two}, which makes these methods unsuitable. To solve this problem, \cite{hong2020two} propose a single-loop two-timescale method, which updates $y$ and $x$ alternately with stepsize $\eta_y$ and $\eta_x$, respectively, designed with different timescales as $\lim_{k\rightarrow \infty}{\eta_x}/{\eta_y}=0$. However, due to the nature of two-timescale updates, it incurs the sub-optimal complexity $\mathcal{O}(\epsilon^{-5})$ \cite{chen2021single}. To further improve the efficiency, \cite{huang2021biadam,khanduri2021near,chen2021single,guo2021randomized} proposed single-loop single-timescale methods, where ${\eta_x}/{\eta_y}$ is a constant. These methods have the complexity of $\tilde{\mathcal{O}}(\epsilon^{-4})$ ($\tilde{\mathcal{O}}$ means omitting logarithmic factors) or better ($\tilde{\mathcal{O}}(\epsilon^{-3})$) to achieve the stationary point. However, all these methods are limited to the bilevel optimization problem with unconstrained lower-level problems and require the upper-level objective
function to be differentiable. They cannot be directly applied when constraints are present in the lower-level optimization, i.e., $\mathcal{Y}\not=\mathbb{R}^{d_2}$,
as the upper-level objective function is naturally non-differentiable \cite{xu2023efficient}.

\begin{table*}[t]
	\centering
	\caption{Several representative hypergradient approximation methods for the lower-level constrained BO problem. 
   (The last column shows iteration numbers to find a stationary point. The gray color is used to highlight the main limitations of the listed algorithms)} 
	\label{tab:constrained_BO_methods}
 \setlength{\tabcolsep}{0.8mm}
 	\begin{tabular}{lcccccc}
		\toprule
		Method  &  $F(x)$&Loop& Timescale &LL. Constraint & Restrictive Conditions& Iterations   \\
		\hline
		AiPOD  \cite{xiao2023alternating}& \cellcolor{lightgray}Smooth&Double&$\times$&\cellcolor{lightgray}Affine sets &Not need &$\tilde{\mathcal{O}}(\epsilon^{-2})$ \\\hline
		IG-AL \cite{9747013}& Nonsmooth&Double&$\times$&\cellcolor{lightgray}Half space &Not need& $\times$ \\
        RMD-PCD  \cite{bertrand2022implicit}&Nonsmooth  & Double&$\times$ &Norm set& \cellcolor{lightgray}$y^*(x)$ is  differentiable &$\times$\\
        JaxOpt\cite{blondel2022efficient}& Nonsmooth&Double&$\times$&Convex set&\cellcolor{lightgray} $y^*(x)$ is differentiable & $\times$ \\
		DMLCBO (Ours) & Nonsmooth & Single&Single & Convex set & Not need & $\tilde{\mathcal{O}}(d_2^2\epsilon^{-4})$ \\
		\bottomrule		
	\end{tabular}
\end{table*}

To solve the lower-level constrained bilevel optimization problem, recently, several methods have been proposed to approximate the hypergradient, as shown in Table \ref{tab:constrained_BO_methods}. Specifically, \cite{xiao2023alternating} reformulate the affine-constrained lower-level problem into an unconstrained problem, and then solve the new lower-level problem and use the implicit differentiable method to approximate the hyper-gradient. However, their convergence analysis only focuses on the affine-constrained problem and cannot be extended to a more general case.  \cite{9747013} solve the inner problem with projection gradient and use the implicit differentiable method to approximate the hyper-gradient for the half-space-constrained BO problem. However, they only give the asymptotic convergence analysis for this special case. Since many methods of calculating the Jacobian of the projection operators have been proposed \cite{martins2016softmax,djolonga2017differentiable,blondel2020fast,niculae2017regularized,vaiter2013local,cherkaoui2020learning}, the explicit or implicit methods can also be used to approximate the hypergradient in the LCBO problems, such as \cite{liu2021towards,bertrand2020implicit,bertrand2022implicit,blondel2022efficient}. However, all these methods are based on restrictive assumptions, namely, where optimality conditions satisfy the differentiability and invertibility conditions, and lack a solid analysis of the convergence rate.  What's worse, these methods can not be utilized to solve the sequential game which is mentioned above. 
Therefore, it is still an open challenge to design a \textit{single-loop single-timescale} method with convergence rate analysis for the lower-level constrained bilevel optimization problems.

To overcome these problems, we propose a novel \textit{single-loop single-timescale} method with a convergence guarantee for the lower-level constrained BO problems. Specifically, instead of using the restive assumptions used in \cite{blondel2022efficient,bertrand2020implicit,bertrand2022implicit}, we leverage the theory of nonsmooth implicit function theorems \cite{clarke1990optimization,bolte2021nonsmooth} to propose a new hypergradient of LCBO. Then, we use the randomized smoothing and Neumann series to further approximate the hypergradient. Using this new hypergradient approximation, we propose a \textit{single-loop single-timescale} algorithm based on the double-momentum method and adaptive step size method to update the lower- and upper-level variables simultaneously. Theoretically, we prove our methods can return a $(\delta,\epsilon)$-stationary point with $\tilde{\mathcal{O}}(d_2^2\epsilon^{-4})$ iterations. The experimental results in two applications demonstrate the effectiveness of our proposed method. 

We summarized our contributions as follows.
\begin{enumerate}
    \item Leveraging the theory of nonsmooth implicit function theorems, we propose a new method to calculate the hypergradient of LCBO without using restrictive assumptions.
    \item We propose a new method to approximate the hypergradient based on randomized smoothing and the Neumann series.  Using this hypergradient approximation, we propose a \textit{single-loop single-timescale} algorithm for the lower-level constrained BO problems.
	\item Existing hypergradient-based methods for solving the lower-level constrained BO problems usually lack theoretical analysis on convergence rate. We prove our methods can return a  $(\delta,\epsilon)$-stationary point with $\tilde{\mathcal{O}}(d_2^2\epsilon^{-4})$ iterations. 
	\item We compare our method with several state-of-the-art methods for the lower-level constrained BO problems on two applications. The experimental results demonstrate the effectiveness of our proposed method. 
\end{enumerate} 

\section{Preliminaries}
\textbf{Notations.} Here, we give several important notations used in this paper.
$\|\cdot\|$ denotes the $\ell_2$ norm for vectors and spectral norm for matrices. $I_d$ denotes a $d$-dimensional identity matrix. $A^{\top}$ denotes transpose of matrix $A$. Given a convex set $\mathcal{X}$, we define a projection operation to $\mathcal{X}$ as $\mathcal{P}_{\mathcal{X}}(x')=\arg\min_{x\in\mathcal{X}}{1}/{2}\|x-x'\|^2$. $co{\mathcal{X}}$ denotes the convex hull of the set $\mathcal{X}$.

\subsection{Lower-level Constrained Bilevel Optimization}
In this paper, we consider the following BO problems where the lower-level problem has convex constraints,
\begin{align}\label{lc_BL}
	\min_{x\in\mathbb{R}^{d_1}} &\ F(x) = f(x,y^*(x)) \\
	s.t. \quad & y^*(x)=\mathop{\arg\min}_{y\in\mathcal{Y}\subseteq\mathbb{R}^{d_2}}g(x,y).\nonumber
\end{align}
 Then, we introduce several mild assumptions on the Problem (\ref{lc_BL}).
\begin{assumption}\label{assump:upper_level}
	The upper-level function $f(x,y)$ satisfies the following conditions:
	\begin{enumerate}
		\item $\nabla_{x}f(x,y)$ is $L_{f}$-Lipschitz continuous w.r.t. $(x,y)\in\mathbb{R}^{d_1}\times\mathbb{R}^{d_2}$ and $\nabla_{y}f(x,y)$ is $L_{f}$-Lipschitz continuous w.r.t $(x,y)\in\mathbb{R}^{d_1}\times\mathbb{R}^{d_2}$, where $L_{f}>0$ .
		\item For any $x\in\mathbb{R}^{d_1}$ and $y\in\mathbb{R}^{d_2}$, we have $\|\nabla_yf(x,y)\|\leq C_{fy}$.
	\end{enumerate}
\end{assumption}
\begin{assumption}\label{assump:lower_level}
	The lower-level function $g(x,y)$ satisfies the following conditions:
	\begin{enumerate}
		\item For any $(x,y)\in\mathbb{R}^{d_1}\times\mathbb{R}^{d_2}$, $g(x,y)$ is twice continuously differentiable in $(x,y)$.
		\item Fix $x$, $\nabla_{y}g(x,y)$ is $L_g$-Lipschitz continuous w.r.t $y$ for some $L_g>0$. 
		\item Fix $x$, for any $y$, $g(x,y)$ is $\mu_g$-strongly-convex in $y$ for some $\mu_g>0$.
		\item  $\nabla_{xy}^2g(x,y)$ is $L_{gxy}$-Lipschitz continuous w.r.t $(x,y)\in\mathbb{R}^{d_1}\times\mathbb{R}^{d_2}$ and $\nabla_{yy}^2g(x,y)$ is $L_{gyy}$-Lipschitz continuous w.r.t $(x,y)\in\mathbb{R}^{d_1}\times\mathbb{R}^{d_2}$, where $L_{gxy}>0$ and $L_{gyy}>0$.
		\item For any $(x,y)\in\mathbb{R}^{d_1}\times\mathbb{R}^{d_2}$, we have $\|\nabla_{xy}^2g(x,y)\|\leq C_{gxy}$.
	\end{enumerate}
\end{assumption}

These assumptions are commonly used in bilevel optimization problems \cite{ghadimi2018approximation,hong2020two,ji2021bilevel,chen2021single,khanduri2021near,guo2021randomized}.

\subsection{Review of Unconstrained Bilevel Optimization Methods}
For the upper-level objective, we can naturally derive the following  gradient w.r.t $x$ using the chain rule (which is defined as hypergradient), 
\begin{align}
	\nabla F(x)=\nabla_x f(x,y^*(x))+(\nabla y^*(x))^{\top}\nabla_yf(x,y^*(x)). \nonumber
\end{align}
The crucial problem of obtaining the hypergradient is calculating $\nabla y^*(x)$. If the lower-level problem is unconstrained, using the implicit differentiation  method and the optimal condition $\nabla_yg(x,y^*(x))=0$, it is easy to show that for a given $x\in\mathbb{R}^{d_1}$, the following equation holds \cite{ghadimi2018approximation,hong2020two,ji2021bilevel,chen2021single,khanduri2021near}
\begin{align}
    \nabla y^*(x)=[\nabla_{yy}^2g(x,y^*(x))]^{-1}\nabla_{yx}^2g(x,y^*(x)).
\end{align}
Substituting $\nabla y^*(x)$ into $\nabla F(x)$, we can obtain the hypergradient. Then, we update $x$ and $y$ alternately using the gradient method.

\subsection{Review of Lower-Level Constrained Bilevel Optimization Problem}

For the constrained lower-level problem, one common method is to use the projection gradient method, which has the following optimal condition,
\begin{align}\label{eqn:y_update}
    y^*(x)=\mathcal{P}_{\mathcal{Y}}(y^*(x)-\eta \nabla_y g(x,y^*(x))),
\end{align}
where $\eta>0$ denotes the step-size. Recently, \cite{blondel2022efficient,bertrand2020implicit,bertrand2022implicit} assume $y^*(x)$ and $\mathcal{P}_{\mathcal{Y}}(\cdot)$ to be differentiable at some special points and use the reverse method 
or implicit gradient method to derive the hypergradient. However, these assumptions are relatively strong and highly limit the usage of these methods. 


\section{Proposed Method}

In this section, we propose a new method to approximate the hypergradient using randomized smoothing that makes convergence analysis possible. Then, equipped with this hypergradient, we propose our \textit{single-loop single-timescale} method to find a stationary point of the lower-level constrained bilevel problem. 

\subsection{Hypergrdient of Lower-level Constrained Bilevel Optimization Problem}

Before we present our method to calculate the hypergradient, we give the following definition of generalized Jacobian and gradient \cite{clarke1990optimization}, which is important in our method.
\begin{definition}
    Given a scaler-valued function $F:\mathbb{R}^n \rightarrow \mathbb{R}$, which is Lipschitz near a given point $x$ of interest. The generalized gradient is defined as:
    \begin{align}
        \partial F(x)=co\{ \lim \nabla F(x_i): x_i\rightarrow x,x_i\not\in S,x_i\not\in\Omega_F\},\nonumber
    \end{align}
    where $\Omega_F$ is the set of points at which $F$ fails to be differentiable and $S$ is any other set of measure zero.   
\end{definition}
\begin{definition}
    Given a vector-valued function $F:\mathbb{R}^n \rightarrow \mathbb{R}^m$, which is Lipschitz near a given point $x$ of interest. Denote the set of points at which $F$ fails to be differentiable by $\Omega_F$ and the usual $m\times n$ Jacobian matrix of partial derivatives by $JF(y)$ whenever $y$ is a point at which the necessary partial derivatives exist. 
    The generalized Jacobian of $F$ at point $x$, denoted as $\partial F(x)$, is the convex hull of all $m\times n$ matrices obtained as the limit of a sequence of the form $JF(x_i)$, where $x_i\rightarrow x$ and $x_i\not\in\Omega_F$. Symbolically, then, one has 
    \begin{align}
        \partial F(x)=co\{ \lim JF(x_i): x_i\rightarrow x,x_i\not\in\Omega_F\}.\nonumber
    \end{align}
\end{definition}
According to \cite{clarke1990optimization}, if we want to use the Jacobian Chain Rule to calculate the generalized gradient $\partial F(x)$ of the LCBO, the most important thing is to ensure that $y^*(x)$ is Lipschitz continuous. Fortunately, under the strongly convex assumption of the lower-level problem in Assumptions \ref{assump:lower_level}, we can obtain the following result.
\begin{lemma}\label{lemma:Lip_y_star}
	Under Assumptions \ref{assump:lower_level}, we have 
	the optimal solution to the lower-level problem is Lipschitz continuous with constant ${L_{g}}/{\mu_g}$.
\end{lemma}

Then, using the Jacobian Chain Rule in \cite{clarke1990optimization} and Lemma \ref{lemma:Lip_y_star}, we have 
\begin{align}
\resizebox{0.9\linewidth}{!}{
    $
    \partial F(x)=\nabla_xf(x,y^*(x))+  (\partial y^*(x))^{\top}\nabla_yf(x,y^*(x)) $}.
\end{align}


To obtain the generalized gradient $\partial F(x)$, the most critical point is to calculate the generalized Jacobian $\partial y^*(x)$. Inspired by the previous methods \cite{blondel2022efficient,bertrand2020implicit,bertrand2022implicit}, we use the Corollary of the Jacobian Chain Rule in \cite{clarke1990optimization} to calculate the generalized Jacobian on Eqn \ref{eqn:y_update} on both sides, since both $y^*(x)$ and $\mathcal{P}_{\mathcal{Y}}(z^*)$ are Lipschitz continuous, where $z^*=y^*(x)-\eta \nabla_y g(x,y^*(x))$. Therefore, we can obtain 
\begin{equation}\label{eqn:partial_y_star}
    \partial y^*(x)v= \partial \mathcal{P}_{\mathcal{Y}}(z^*)\partial\left(y^*(x)-\eta\nabla_y g(x,y^*(x))\right)v.
\end{equation}
where $v\in\mathbb{R}^{d_1}$. Using the same method, we have 
\begin{align}\label{eqn:partial_y_star_v}
    &\partial\left(y^*(x)-\eta\nabla_y g(x,y^*(x))\right)v
    =\partial y^*(x)v\nonumber\\
    &-\eta   (\nabla_{yx}^2 g(x,y^*(x))- \nabla_{yy}^2g(x,y^*(x))\partial y^*(x))v.
\end{align}
Substituting Equation (\ref{eqn:partial_y_star_v}) into formula (\ref{eqn:partial_y_star}), we have 
\begin{align}
    \partial y^*(x)v=&\partial \mathcal{P}_{\mathcal{Y}}(z^*)(\partial y^*(x)v-\eta   (\nabla_{yx}^2 g(x,y^*(x))\nonumber\\
    &+ \nabla_{yy}^2g(x,y^*(x))\partial y^*(x))v).
\end{align}
Suppose there exist $A\in\partial y^*(x)$ and $H\in\partial\mathcal{P}_{\mathcal{Y}}(z^*)$, which makes the following equality hold,
\begin{align}
\resizebox{0.9\linewidth}{!}{
    $Av=H\cdot\left(A-\eta\nabla_{yx}^2 g(x,y^*(x))- \eta\nabla_{yy}^2g(x,y^*(x))A\right)v.$}
\end{align}
Then, rearranging the above equality, we have 
\begin{align}
    &\left[I_{d_2}-H\cdot\left(I_{d_2}- \eta\nabla_{yy}^2g(x,y^*(x))\right)\right]A\nonumber\\
    =&-\eta H\cdot\nabla_{yx}^2 g(x,y^*(x))
\end{align}
To solve the above equation, the challenge is to ensure that $I_{d_2}-H\cdot\left(I_{d_2}- \eta\nabla_{yy}^2g(x,y^*(x))\right)$ is invertible. Since $\mathcal{P}_{\mathcal{Y}}(z^*)$ is non-expensive, we have $\|H\|\leq 1$. Since $g$ is strongly convex, setting $\eta<{1}/{\mu_g}$, we have $\|I_{d_2}-\eta\nabla_{yy}^2g(x,y^*(x))\|< 1$. Therefore, we have $\|H\cdot\left(I_{d_2}- \eta\nabla_{yy}^2g(x,y^*(x))\right)\|< 1$ and $H(I_{d_2}- \eta\nabla_{yy}^2g(x,y^*(x)))$ is nonsingular. Therefore, we can obtain 
\begin{align}
   A
   =&-\eta\left[I_{d_2}-H(I_{d_2}-\eta\nabla_{yy}^2g(x,y^*(x)))\right]^{-1}\nonumber\\
   &\cdot H \nabla_{yx}^2g(x,y^*(x)) 
\end{align}
Then, substituting $A$ into the generalized gradient, we can obtain the following subset of the generalized gradient, 
\begin{align}
    \bar{\partial} F(x)=&\{h|h = \nabla_x f(x,y^*(x)) 
    -\eta \nabla_{xy}^2g(x,y^*(x)) H^{\top}\nonumber\\ 
    &\cdot\left[I_{d_2}-(I_{d_2}-\eta\nabla_{yy}^2g(x,y^*(x)))\cdot H^{\top}\right]^{-1}\nonumber\\
    &\cdot\nabla_yf(x,y^*(x)), H\in\partial\mathcal{P}_{\mathcal{Y}}(z^*) \}.
\end{align}
Obviously, we have $\bar{\partial} F(x)\subset \partial F(x)$. Note that we can find that the hypergradient used in \cite{blondel2022efficient,bertrand2020implicit,bertrand2022implicit} can be viewed as a specific element in our hypergradient. However, these methods need the projection operator to be differentiable, which is impractical in most cases. Therefore, we can conclude that our method can obtain the hypergradient in a more general case. 

After obtaining the hypergradient, our next step is to design an algorithm to find the point $x$ satisfying the condition 
\begin{align}
    \min \{\|h\|:h\in\bar{\partial} F(x) \}\leq \epsilon,
\end{align}
such that $x$ is a $\epsilon$-Clarke stationary point \cite{clarke1990optimization} which satisfies the condition 
\begin{align}
    \min \{\|h\|:h\in\partial F(x) \}\leq \epsilon.
\end{align}
However, \cite{zhang2020complexity} point out that finding an $\epsilon$ stationary points in nonsmooth nonconvex optimization can not be achieved by any finite-time algorithm given a fixed tolerance $\epsilon\in[0,1)$. This suggests the need to rethink the definition of stationary points.

Inspired by \cite{zhang2020complexity,lin2022gradient}, we propose to consider the following $\delta$-approximation generalized Jacobian.
\begin{definition}
    Given a point $z\in\mathbb{R}^{d_2}$ and $\delta>0$, the $\delta$-approximation generalized Jacobian of a Lipschitz function of $\mathcal{P}_{\mathcal{Y}}(\cdot)$ at $z$ is given by $\partial_{\delta}\mathcal{P}_{\mathcal{Y}}(z):=co\left(\bigcup_{z'\in\mathbb{B}_{\delta}(z)}\partial \mathcal{P}_{\mathcal{Y}}(z')\right)$. 
\end{definition}
The above approximation generalized Jacobian of $\mathcal{P}_{\mathcal{Y}}(\cdot)$ at $z$ is the convex hull of all generalized Jacobians at points in a $\delta$-ball around $z$, and it can be viewed as the extension of Goldstein subdifferential \cite{zhang2020complexity,lin2022gradient, goldstein1977optimization}. Then, substituting $\partial_{\delta}\mathcal{P}_{\mathcal{Y}}(z)$ into our hypergradient $\bar{\partial} F(x)$, we can obtain the following approximation,
\begin{align}\label{eqn:bar_partial_delta_F}
    \bar{\partial}_{\delta} F(x)=&\{h|h = \nabla_x f(x,y^*(x)) 
    -\eta \nabla_{xy}^2g(x,y^*(x)) H^{\top}\nonumber\\ 
    &\cdot\left[I_{d_2}-(I_{d_2}-\eta\nabla_{yy}^2g(x,y^*(x)))\cdot H^{\top}\right]^{-1}\nonumber\\
    &\cdot\nabla_yf(x,y^*(x)), H\in\partial_{\delta}\mathcal{P}_{\mathcal{Y}}(z^*) \}
\end{align}
Equipping with this hypergradient, even though we are unable to find an $\epsilon$-stationary point, one could hope to find a point that is close to an $\epsilon$-stationary point. This motivates us to adopt the following more refined notion
\begin{definition}
    A point $x$ is called $(\delta,\epsilon)$-stationary if $\min\left\{\|h\|:h\in\bar{\partial}_{\delta} F(x)\right\}\leq \epsilon$.
\end{definition}
Note that if we can find a point $x'$ at most distance $\delta$ away from $x$ such that $x'$ is $\epsilon$-stationary, then we know $x$ is $(\delta,\epsilon)$-stationary. However, the contrary is not true. We also have the following result
\begin{lemma}\label{lemma:converge_p_delta}
    The set $\partial_{\delta} \mathcal{P}_{\mathcal{Y}}(z)$ converges as $\delta\downarrow0$ as 
    $
        \lim_{\delta\downarrow0}\partial_{\delta} \mathcal{P}_{\mathcal{Y}}(z)=\partial \mathcal{P}_{\mathcal{Y}}(z)
    $.
\end{lemma}

This result enables an intuitive framework for transforming the non-asymptotic analysis of convergence to $(\delta,\epsilon)$-stationary points to classical asymptotic results for finding $\epsilon$-stationary points. Thus, we conclude that finding a $(\delta,\epsilon)$ stationary point is a reasonable optimality condition for the lower-level constrained bilevel optimization problem.

\subsection{Randomized Smoothing}

Inspired by the strong ability of randomized smoothing to deal with nonsmooth problems, in this subsection, we use this method to handle the non-smoothness of the projection operator.
Given a non-expansive projection operator \cite{moreau1965proximite} $\mathcal{P}_{\mathcal{Y}}(z)$ and uniform distribution $\mathbb{P}$ on a unit ball in $\ell_2$-norm, we define the smoothing function as $\mathcal{P}_{\mathcal{Y}\delta}(z)=\mathbb{E}_{u\sim\mathbb{P}}[\mathcal{P}_{\mathcal{Y}}(z+\delta u)]$. Then, we have the following proposition. 
\begin{proposition}\label{prop1}
    Let $\mathcal{P}_{\mathcal{Y}\delta}(z)=\mathbb{E}_{u\sim\mathbb{P}}[\mathcal{P}_{\mathcal{Y}}(z+\delta u)]$ where $\mathbb{P}$ is a uniform distribution on a unit ball in $\ell_2$-norm. Since that $\mathcal{P}_{\mathcal{Y}}$ is non-expansive and each element of $\mathcal{P}_{\mathcal{Y}}$ is Lipschitz continuous with constant $L_{p}$, we have
    $\mathcal{P}_{\mathcal{Y}\delta}(z)$ is differentiable and $1$-Lipschitz continuous with ${cd_2L_p}/{\delta}$-Lipschitz gradient, where $c>0$ is constant.
\end{proposition}

Using this randomized smoothing function to replace the approximation generalized Jacobian in Eqn (\ref{eqn:bar_partial_delta_F}), we can approximate the hypergradient as follows,
\begin{align}
    &\nabla F_{\delta}(x)\nonumber\\
    =&\nabla_x f(x,y^*(x))- \eta \nabla_{xy}^2g(x,y^*(x)) \nabla\mathcal{P}_{\mathcal{Y}\delta}(z^*)^{\top}\nonumber\\
    &\cdot\left[I_{d_2}-(I_{d_2}-\eta\nabla_{yy}^2g(x,y^*(x)))\nabla\mathcal{P}_{\mathcal{Y}\delta}(z^*)^{\top}\right]^{-1}\nonumber\\
    &\cdot\nabla_yf(x,y^*(x)).\nonumber
\end{align}

For this hypergradient estimation, we have the following conclusion.
\begin{lemma}\label{lemma:lip_f_x}
    Under Assumptions \ref{assump:upper_level}, \ref{assump:lower_level}, we have $\nabla F_{\delta}(x)$ is Lipschitz continuous w.r.t $x$.
\end{lemma}
Lemma \ref{lemma:lip_f_x} indicates that we can use the traditional analysis framework to discuss the convergence performance using $\nabla F_{\delta}(x)$. However, as we discussed above, our purpose is to find the $(\delta,\epsilon)$-stationary points of LCBO. Therefore, a new challenge arises, namely, to show the relation between $\nabla F_{\delta}(x)$ and $\bar{\partial}_{\delta}F(x)$. Here, we first give the relation between the $\delta$-approximation generalized Jacobian and the Jacobian of the randomized smoothing function as follows.
\begin{proposition}\label{prop2}
    We have $\nabla\mathcal{P}_{\mathcal{Y}\delta}(z)\in\partial_{\delta} \mathcal{P}_{\mathcal{Y}}(z)$ for any $z\in\mathbb{R}^{d_2}$.
\end{proposition}
Using this result, we have $\nabla F_{\delta}(x)\in\bar{\partial}_{\delta}F(x)$. This resolves an important question and forms the basis for analyzing our method, which means that once we find a point satisfying the condition $\|\nabla F_{\delta}(x)\|\leq \epsilon$, then it is a $(\delta, \epsilon)$-stationary point.

\subsection{ Approximation of Hypergradient}


In this subsection, we discuss how to approximate the hypergradient $\nabla F_{\delta}(x)$ in an efficient method. 

Since obtaining the optimal solution $y^*(x)$ is usually time-consuming,  one proper method is to replace $y^*(x)$ with $y$ as an approximation as follows,
\begin{align}
    &\nabla f_{\delta}(x,y)
    =\nabla_x f(x,y)- \eta \nabla_{xy}^2g(x,y) \nabla\mathcal{P}_{\mathcal{Y}\delta}(z)^{\top}\nonumber\\
    &\cdot\left[I_{d_2}-(I_{d_2}-\eta\nabla_{yy}^2g(x,y))\nabla\mathcal{P}_{\mathcal{Y}\delta}(z)^{\top}\right]^{-1}\nabla_yf(x,y).\nonumber
\end{align}
where $z=y-\eta \nabla_y g(x,y)$. Then, we can use the Neumann series \cite{ghadimi2018approximation,meyer2000matrix} to efficiently approximate the matrix inverse, since we have $\|(I_{d_2}-\eta\nabla_{yy}^2g(x,y))\nabla\mathcal{P}_{\mathcal{Y}\delta}(z)^{\top}\|< 1$ if $\eta<1/\mu_g$. Therefore, we can obtain the approximation of the hypergradient as follows,
\begin{align}
    &\bar{\nabla} f_{\delta}(x,y)
    =\nabla_x f(x,y)- \eta  \nabla_{xy}^2g(x,y) \nabla\mathcal{P}_{\mathcal{Y}\delta}(z)^{\top}\nonumber\\
    &\cdot \sum_{i=0}^{Q-1}\left((I_{d_2}-\eta\nabla_{yy}^2g(x,y))\nabla\mathcal{P}_{\mathcal{Y}\delta}(z)^{\top}\right)^i\nabla_yf(x,y).\nonumber
\end{align}

In addition, since calculating $\nabla \mathcal{P}_{\mathcal{Y\delta}}(z)$ is impractical,  we can use the following unbiased estimator of the gradient $\nabla \mathcal{P}_{\mathcal{Y\delta}}(z)$ as a replacement,
\begin{align}
	\bar{H}(z;u)=\sum_{i=1}^{d_2}\frac{1}{2\delta}\left( \mathcal{P}_{\mathcal{Y}}(z+\delta u_i)- \mathcal{P}_{\mathcal{Y}}(z-\delta u_i) \right)u_i^{\top},\nonumber
\end{align}
where $u_i\sim\mathbb{P}$. Note that, in the program, we can calculate $\bar{H}$ in parallel, thereby reducing the actual running time. For $\bar{H}(z;u)$, we have the following conclusion.
\begin{lemma}\label{lemma:bound_of_nablaP}
    We have  $\mathbb{E}\left[\bar{H}(z;u)\right]=\nabla\mathcal{P}_{\mathcal{Y}\delta}(z)$ and $\mathbb{E}\left[\left\|\bar{H}(z;u)- \nabla\mathcal{P}_{\mathcal{Y}\delta}(z)\right\|^2\right]\leq16\sqrt{2\pi}d_2L_p^2
    $.
\end{lemma}
 To further reduce the complexity caused by calculating multiple Jacobian-vector products, we can introduce an additional stochastic layer on the finite sum. Specifically, assume we have a parameter $Q>0$ and a collection of independent samples $\bar{\xi}:=\{u^0,\cdots,u^{c(Q)}\}$, where $c(Q)\sim\mathcal{U}\left\{0,\cdots,Q-1\right\}$. Then, we can approximate the gradient as follows,
\begin{align}\label{eqn:stoc_hyper_grad}
   &\bar{\nabla} f_{\delta}(x,y;\bar{\xi})
    =\nabla_x f(x,y)- \eta Q \nabla_{xy}^2g(x,y) \bar{H}(z;u^0)^{\top}\nonumber\\
    &\cdot\prod_{i=1}^{c(Q)}\left( (I_{d_2}-\eta\nabla_{yy}^2g(x,y)) \bar{H}(z;u^i)^{\top}\right)\nabla_yf(x,y),
\end{align}
where we have used the the convention $\prod_{i=1}^0 A=I$. 
We can conclude that the bias of the gradient estimator $\bar{\nabla} f_{\delta}(x,y;\bar{\xi})$ decays exponentially fast with $Q$, as summarized below: 
\begin{lemma}\label{lemma:bounds_of_hypergrad}
	Under Assumptions \ref{assump:upper_level}, \ref{assump:lower_level}and Lemma \ref{lemma:bound_of_nablaP}, setting $\frac{1}{\mu_g}(1-\frac{1}{4(2\pi)^{1/4}\sqrt{d_2}L_p})\leq\eta<\frac{1}{\mu_g}$, for any $x\in\mathbb{R}^{d_1}$, $y\in\mathcal{Y}$, we have 
    $
        \left\|\nabla f_{\delta}(x,y)-\mathbb{E}[\bar{\nabla} f_{\delta}(x,y;\bar{\xi})]\right\|\leq\frac{ C_{gxy}C_{fy}}{\mu_g}(1-\eta\mu_g)^Q. 
    $
	Furthermore, the variance of  $\bar{\nabla}f_{\delta}(x,y;\bar{\xi})$ is bounded as 
    $
		\mathbb{E}\left[\left\| \bar{\nabla}f_{\delta}(x,y;\bar{\xi}) -\mathbb{E}\left[\bar{\nabla}f_{\delta}(x,y;\bar{\xi})\right]\right\|^2\right]\leq\sigma_f(d_2),
    $
	where $\sigma_f(d_2)$ is defined in Appendix I.
\end{lemma}

\begin{algorithm}[!ht]
\caption{DMLCBO} \label{alg:DMSCBO}
\renewcommand{\algorithmicrequire}{\textbf{Input:}}
\renewcommand{\algorithmicensure}{\textbf{Output:}}
\begin{algorithmic}[1] 
\REQUIRE  Initialize $x_1\in\mathcal{X}$, $y_{1}\in\mathcal{Y}$, $v_{1}=\nabla_y g(x_{1},y_{1}) $, $w_1=\bar{\nabla} f_{\delta}(x_{1},y_{1};\bar{\xi}_{1})$, $\eta_k$, $\tau$, $\gamma$, $\beta$, $\alpha$, $Q$ and $\eta$. 
\FOR{$k=1,\cdots,K$}
\STATE Update ${x}_{k+1}=x_k-\frac{\eta_k\gamma}{\mathcal{P}_{[1/c_u,1/c_l]}(\sqrt{m_{2,k}}+G_0)}w_{k}$.		

\STATE Update $y_{k+1}=(1-\eta_k)y_{k}+\eta_k\mathcal{P}_{\mathcal{Y}}(y_{k}-\frac{\tau}{\mathcal{P}_{[1/c_u,1/c_l]}(\sqrt{m_{1,k}}+G_0)}  v_{k})$
\STATE Calculate the hyper-gradient $ \bar{\nabla} f_{\delta}(x_{k+1},y_{k+1};\bar{\xi}_{k+1})$ according to Eqn. (\ref{eqn:stoc_hyper_grad}).
\STATE Update $w_{k+1}=(1-\alpha)w_{k}+\alpha\bar{\nabla} f_{\delta}(x_{k+1},y_{k+1};\bar{\xi}_{k+1})$.
\STATE Update $v_{k+1}=(1-\beta)v_{k}+\beta\nabla_yg(x_{k+1},y_{k+1})$.
\ENDFOR
\ENSURE $x_{r}$ where $r\in\{1,\cdots,K\}$ is uniformly sampled.
\end{algorithmic}
\end{algorithm}

\subsection{Double-momentum Method for Lower-level Constrained Bilevel Optimization}

Equipped with the hypergradient $\bar{\nabla} f_{\delta}(x,y;\bar{\xi})$, our next endeavor is to design a \textit{single-loop} \textit{single-timescale} algorithm to solve the constrained bilevel optimization problem (\ref{lc_BL}). Our main idea is to adopt the double-momentum-based method and adaptive step-size method developed in \cite{huang2021biadam,khanduri2021near,shi2022gradient}. Our algorithm is summarized in Algorithm \ref{alg:DMSCBO}. Since we use the double-momentum method to solve the lower-level constrained bilevel optimization problem, we denote our method as DMLCBO.

Define $\alpha\in(0,1)$ and $\beta\in(0,1)$. For the lower-level problem, we can utilize the following projected gradient method with the momentum-based gradient estimator and adaptive step size to update $y$,
\begin{align}
        \hat{y}_{k+1} =& \mathcal{P}_{\mathcal{Y}}(y_{k}-\frac{\tau}{\mathcal{P}_{[1/c_u,1/c_l]}(\sqrt{m_{1,k}}+G_0)}  v_{k}),\nonumber\\
	y_{k+1}=&(1-\eta_k)y_{k}+\eta_k\hat{y}_{k+1},\nonumber\\
 v_{k+1}=&(1-\beta)v_{k}+\beta\nabla_yg(x_{k},y_{k}), \nonumber
\end{align}
where $\eta_k>0$, $\tau>0$. $G_0>0$ is used to avoid to prevent the molecule from being equal to 0. Here, we initialize  $v_1=\nabla_{y}g(x_{1},y_{1})$. Similarly, for the upper-level problem, we can utilize the following gradient method with the momentum-based gradient estimator and adaptive step size to update $x$,
\begin{align}
	{x}_{k+1}=&x_k-\frac{\eta_k\gamma}{\mathcal{P}_{[1/c_u,1/c_l]}(\sqrt{m_{2,k}}+G_0)}w_{k},\nonumber\\
 w_{k+1}=&(1-\alpha)w_{k}+\alpha\bar{\nabla} f_{\delta}(x_k,y_k;\bar{\xi}_k), \nonumber
\end{align}
and we initialize $w_{1}= \bar{\nabla} f_{\delta}(x_{1},y_{1};\bar{\xi}_{1})$. Note we set $m_{i,k+1}=0.99\cdot m_{1,k}+0.01\cdot\mathcal{G}^2$, where $\mathcal{G}$ denotes the gradient estimation, which is used in Adam \cite{kingma2014adam}.

\section{Convergence Analysis}
In this section, we discuss the convergence performance of our DMLCBO (All the detailed proofs are presented in our Appendix). We follow the theoretical analysis framework in \cite{huang2021biadam,huang2020accelerated,shi2022gradient,khanduri2021near} (For easy understanding, we also provide a route map of the analysis in Appendix \ref{section:route_map}).

\begin{figure*}[t]
    
    \begin{subfigure}{0.245\textwidth}
        \centering
        \includegraphics[width=1.4in]{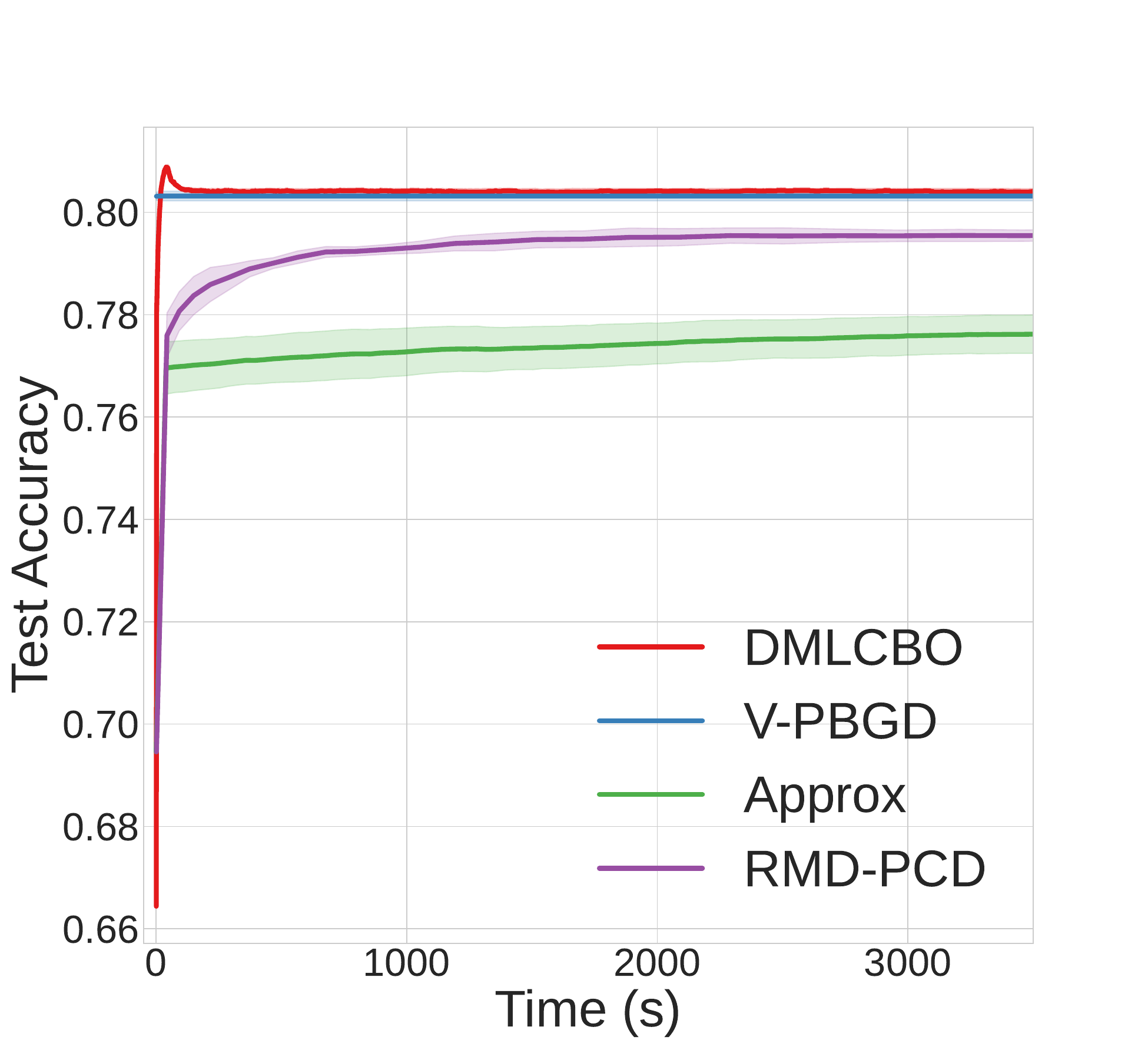}
        \caption{CodRNA}
    \end{subfigure}
    \begin{subfigure}{0.245\textwidth}
        \centering
        \includegraphics[width=1.4in]{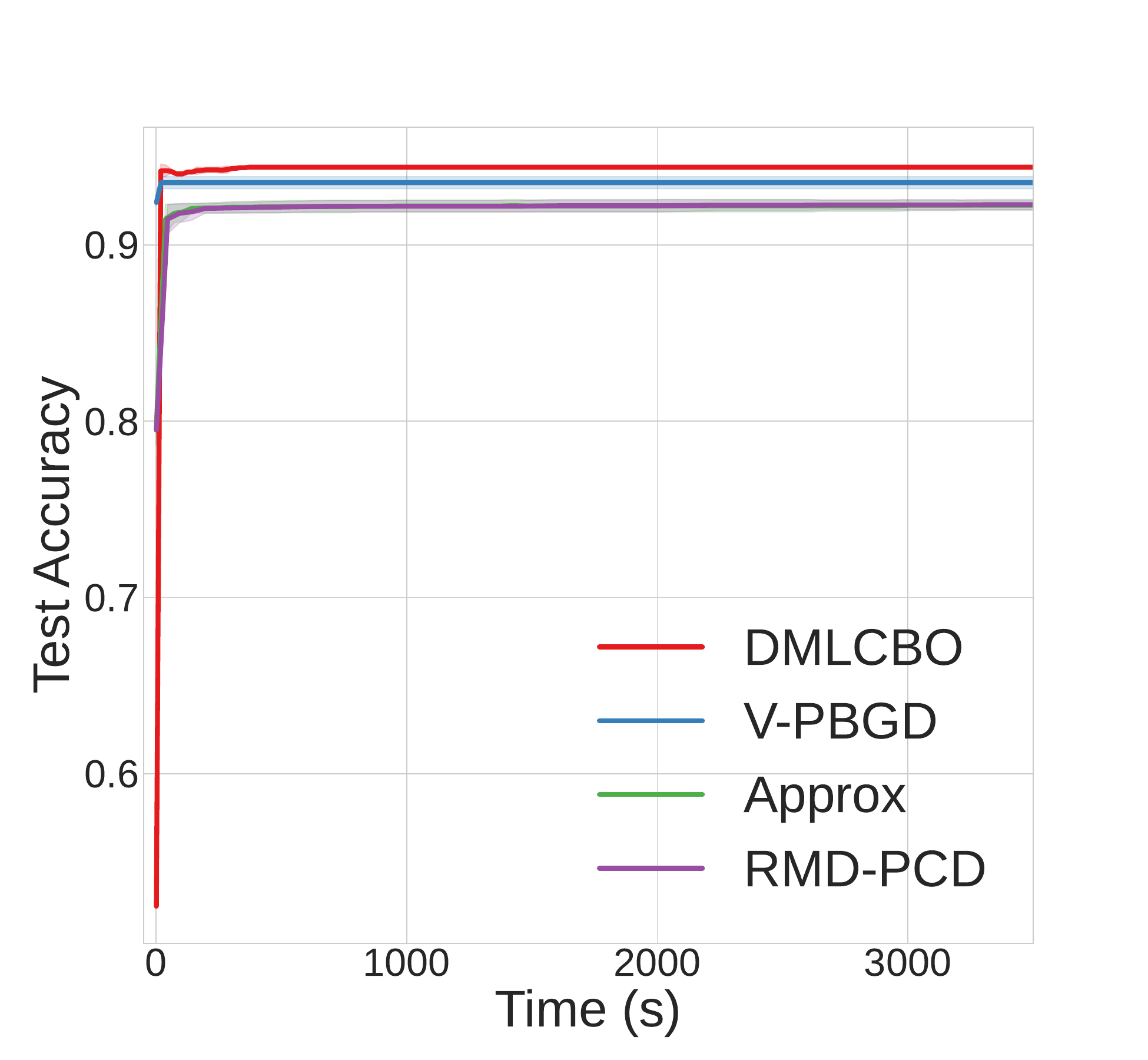}
        \caption{MNIST 6vs9}
    \end{subfigure}	
    \begin{subfigure}{0.245\textwidth}
        \centering
        \includegraphics[width=1.4in]{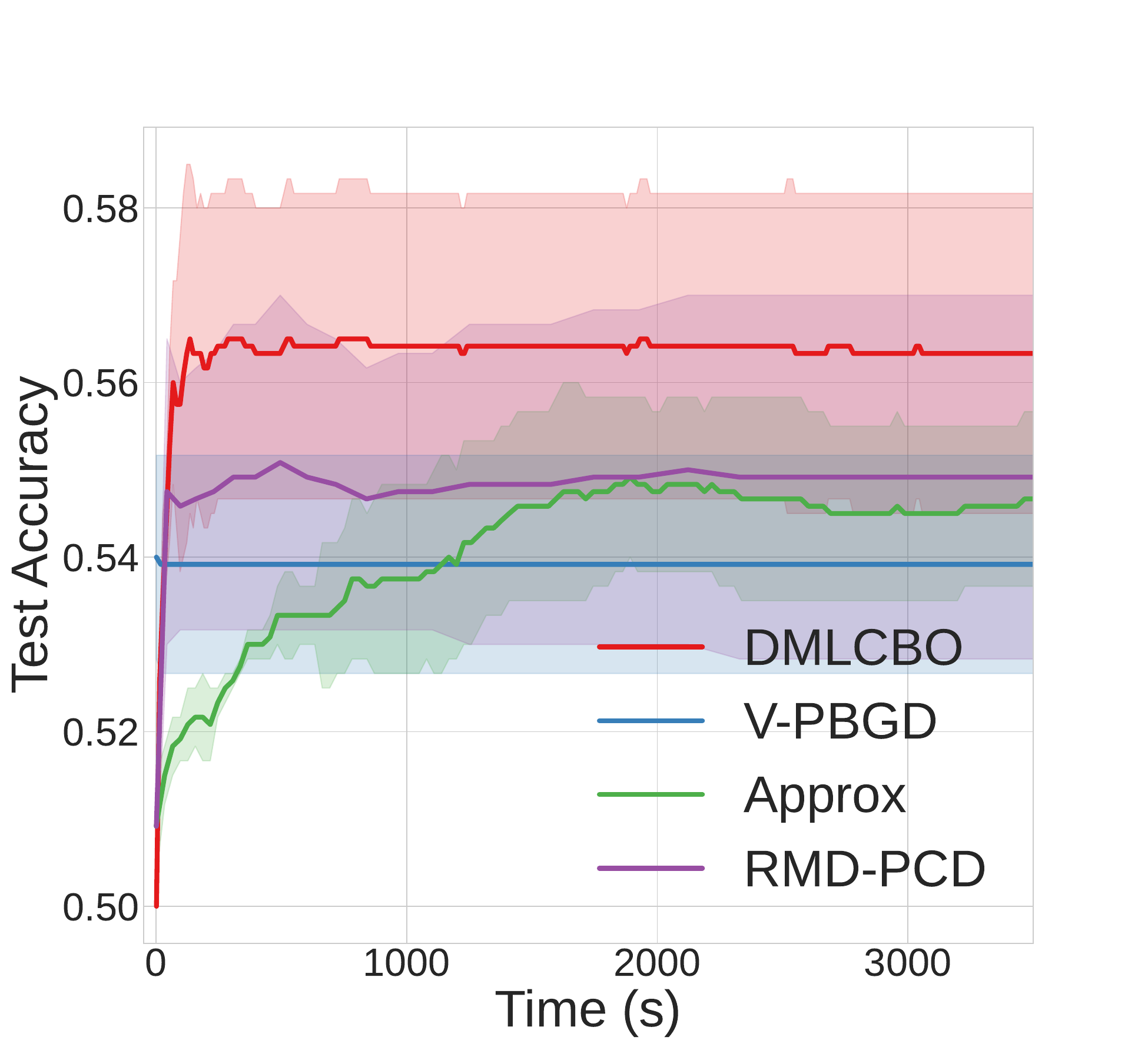}
        \caption{Madelon}
    \end{subfigure}
    \begin{subfigure}{0.245\textwidth}
        \centering
        \includegraphics[width=1.4in]{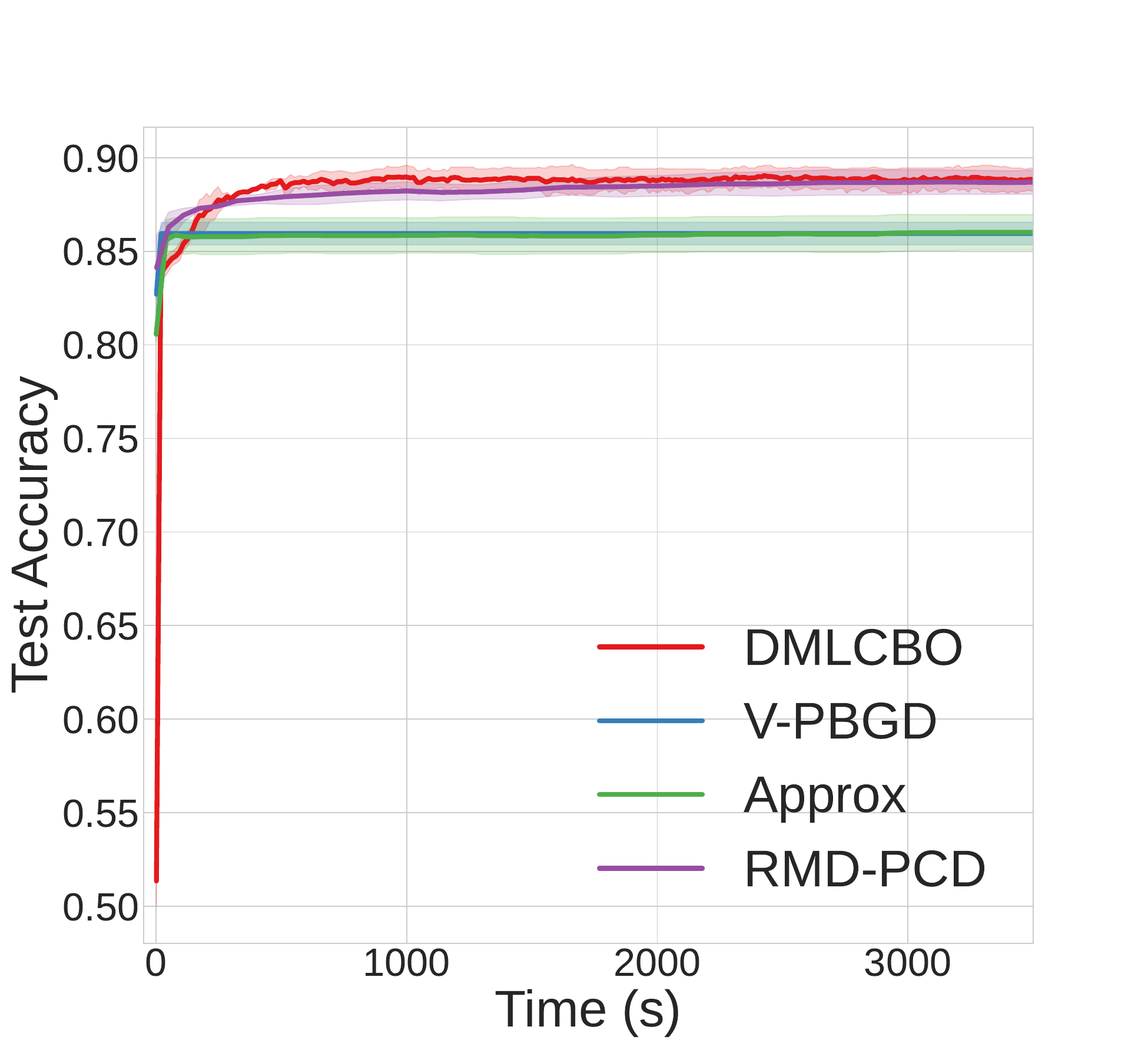}
        \caption{FashionMNIST 1vs7}
    \end{subfigure}
    \caption{Test accuracy against training time of all the methods in data hyper-cleaning.}
    \label{fig:data_reweight_time_vs_acc}
\end{figure*}

\begin{table*}[t]
\setlength{\tabcolsep}{1.4mm}
\caption{Test accuracy (\%) with standard variance of all the methods in data hyper-cleaning. (Higher is better.)}
\label{table:data_reweight}
\centering
\begin{tabular}{lcccc}
\toprule
Datasets        & DMLCBO           &  V-PBGD       & Approx           &  RMD-PCD             \\ \hline
CodRNA     & $\textbf{80.42} \pm 0.05$   & $80.32 \pm 0.09$ & $77.67 \pm 0.38$ & $79.43 \pm 0.07$ \\
MNIST 6vs9      & $\textbf{94.24} \pm 0.01$ & $93.53 \pm 0.34$ & $92.29 \pm 0.27$ & $92.33 \pm 0.29$ \\
Madelon    & $\textbf{56.33} \pm 1.83$ & $53.91 \pm 1.25$ & $54.75 \pm 1.08$ & $ 55.00 \pm 2.17$\\
FashionMNIST 1vs7  & $\textbf{88.85} \pm 0.60$ & $85.95 \pm 0.60$ & $86.03 \pm 0.96$  & $88.93 \pm 0.68$ \\
\bottomrule
\end{tabular}
\end{table*}

\begin{figure*}[t]
    \begin{subfigure}{0.245\textwidth}
        \centering
        \includegraphics[width=1.4in]{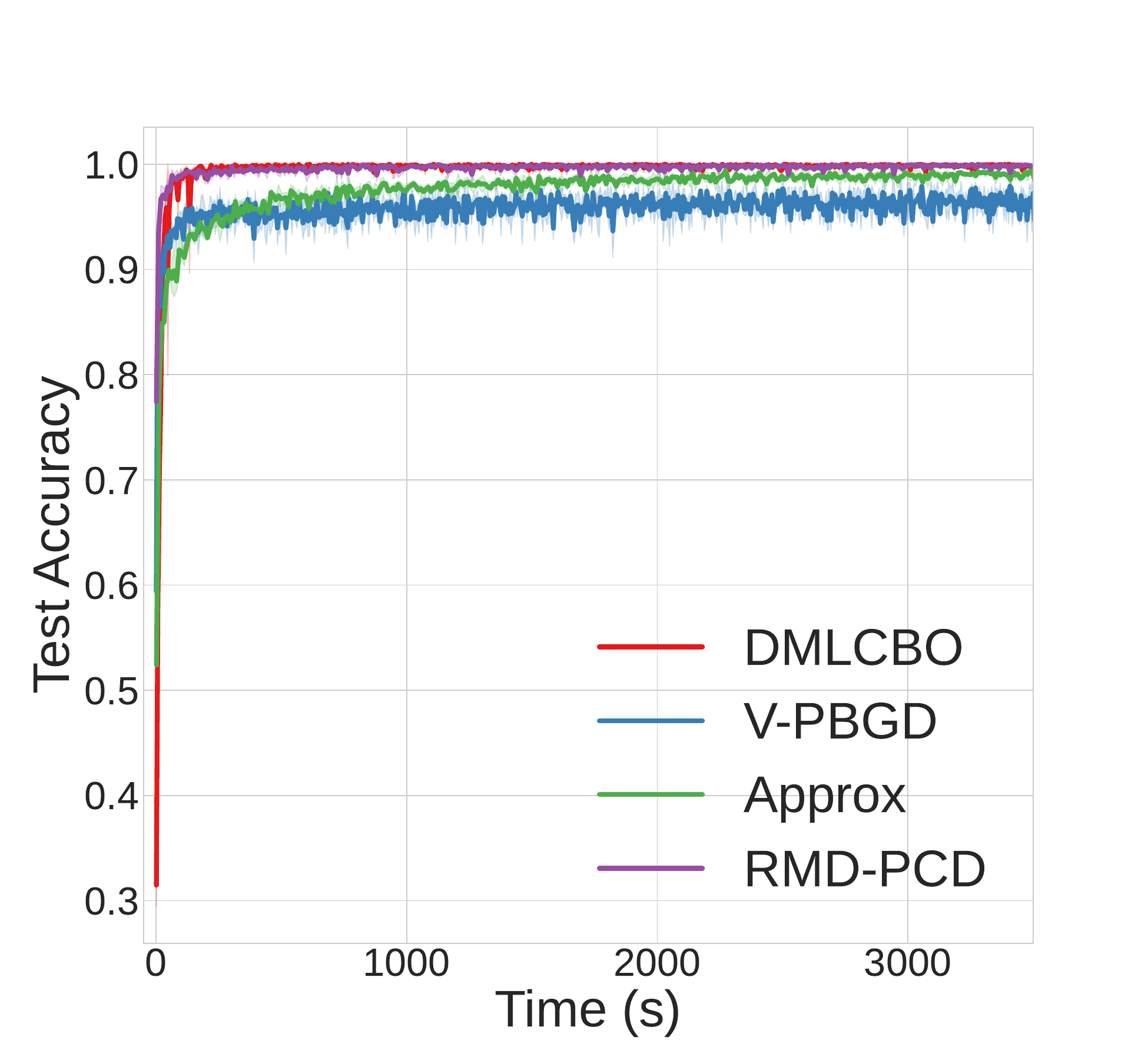}
        \caption{Omniglot 5 way 5 shot }
    \end{subfigure}
    \begin{subfigure}{0.245\textwidth}
        \centering
        \includegraphics[width=1.4in]{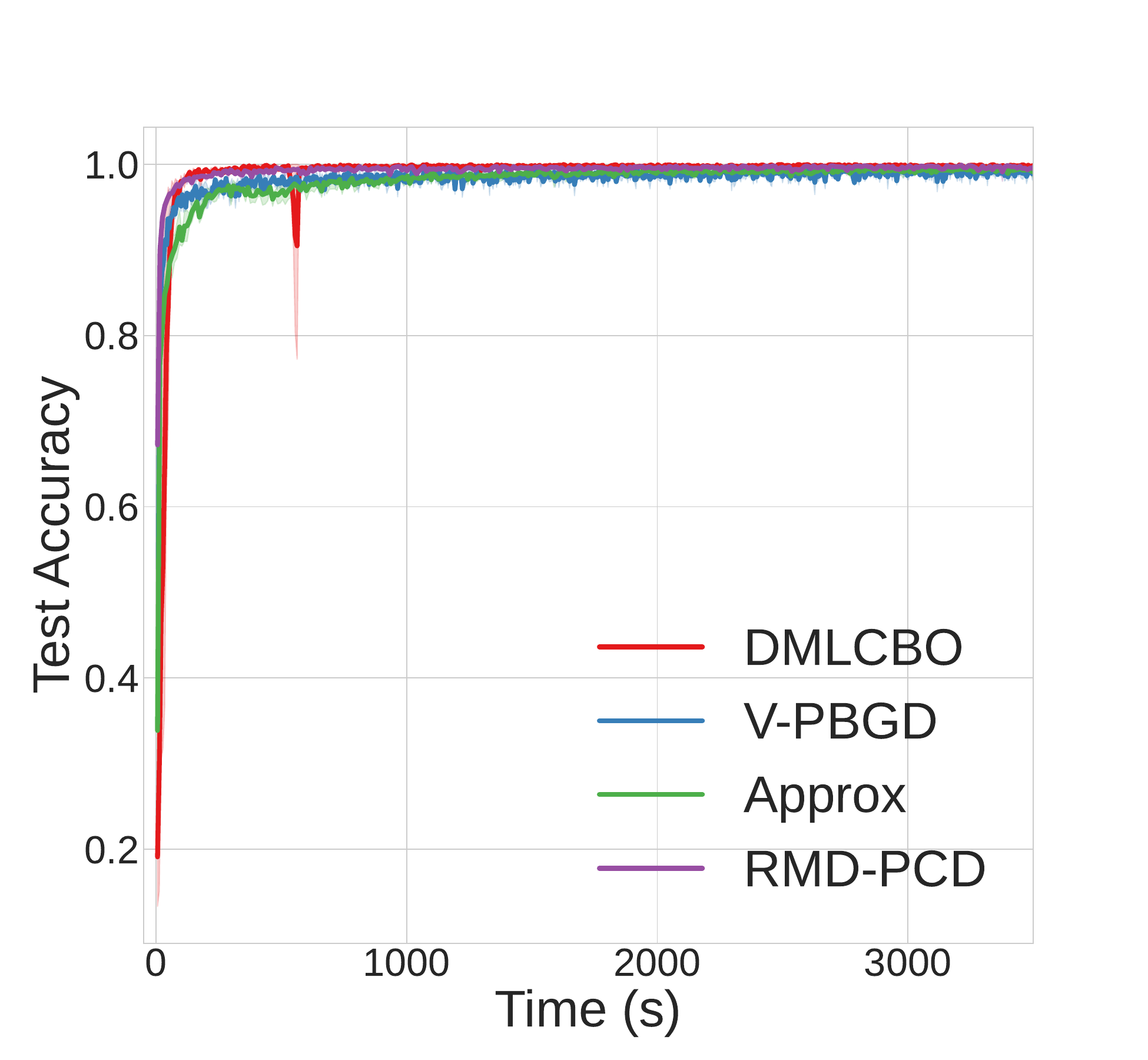}
        \caption{Omniglot 10 way 5 shot}
    \end{subfigure}
    \begin{subfigure}{0.245\textwidth}
        \centering
        \includegraphics[width=1.4in]{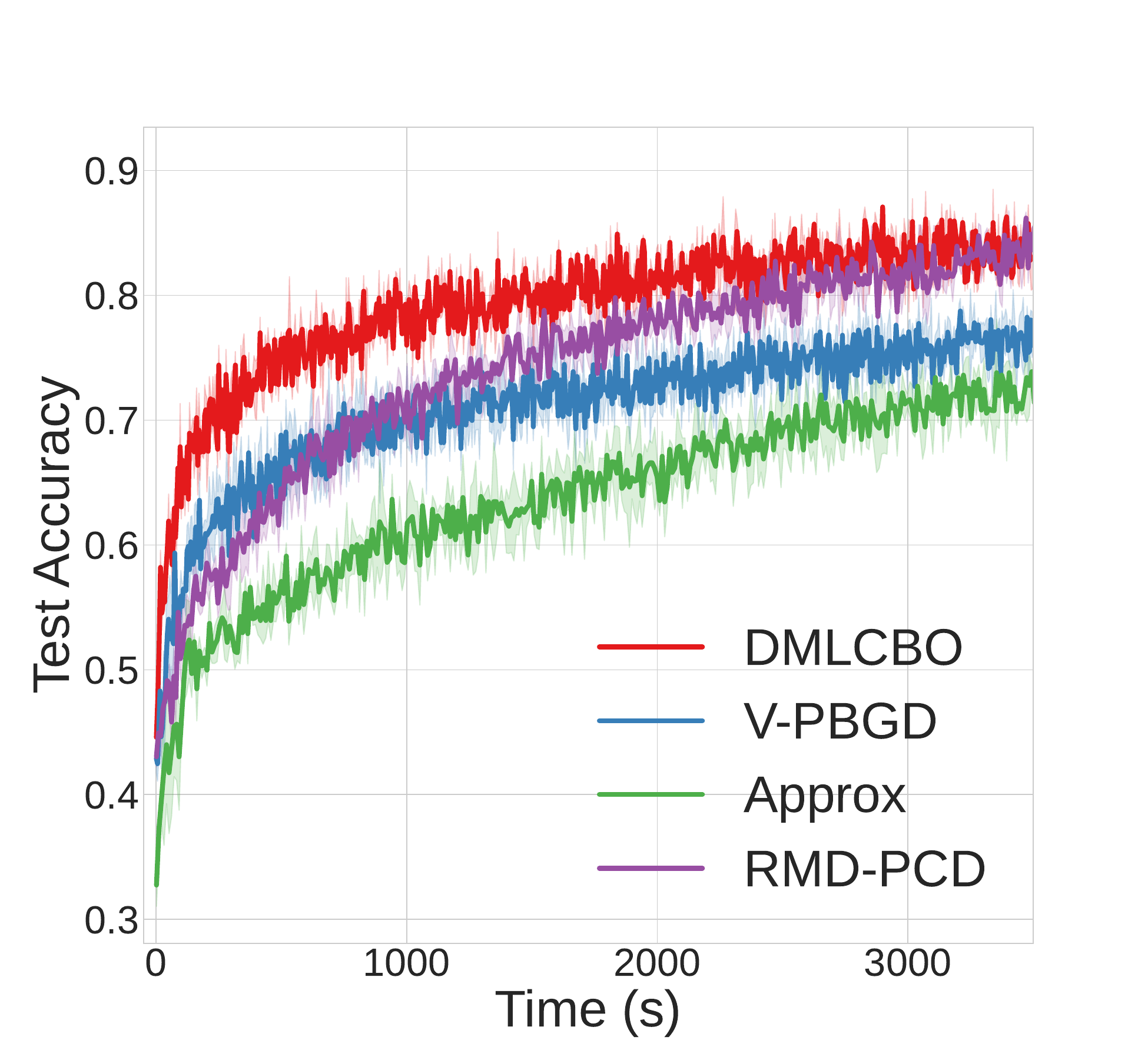}
        \caption{FC100 5 way 5 shot}
    \end{subfigure}	
    \begin{subfigure}{0.245\textwidth}
        \centering
        \includegraphics[width=1.4in]{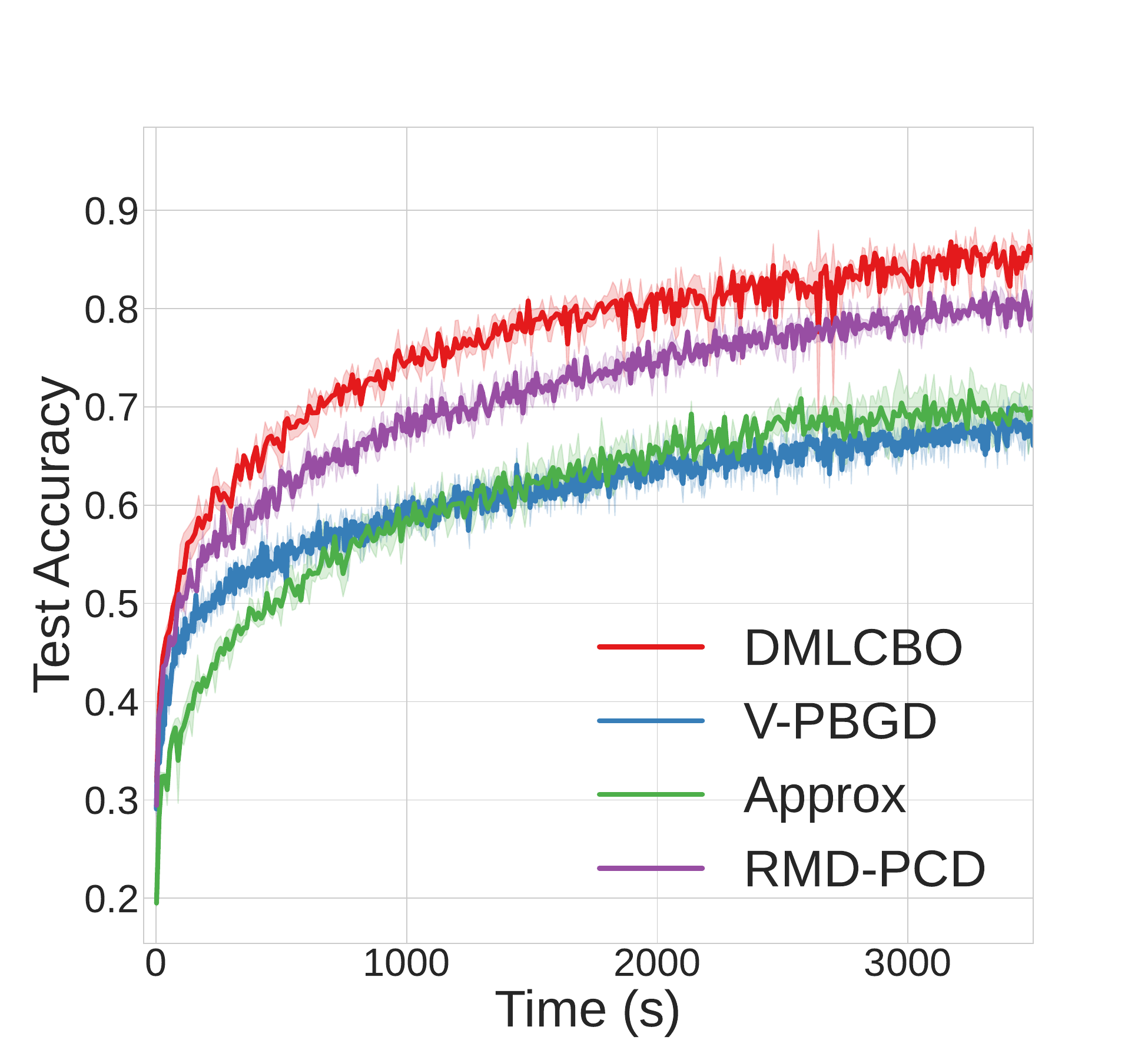}
        \caption{FC100 10 way 5 shot}
    \end{subfigure}
    \caption{Test accuracy against training time of all the methods in meta-learning. }
    \label{fig:poison_attack_time_vs_acc}
\end{figure*}

\begin{table*}[t]
\caption{Test accuracy (\%) with standard variance in meta-learning. (Higher is better.)}
\label{table:poison_attack_acc}
\centering
\setlength{\tabcolsep}{1.8mm}

\begin{tabular}{llcccc}
\toprule
Datasets  & Problem           & DMLCBO           & V-PBGD         & Approx           & RMD-PCD             \\ \hline
Omniglot& 5 way 5 shot      & $\textbf{99.97} \pm 0.04$ & $95.67 \pm 1.02$   & $99.41 \pm 0.25$ & $99.69 \pm 0.37$  \\
Omniglot& 10 way 5 shot & $\textbf{99.81} \pm 0.19$ & $98.43 \pm 0.16$ & $99.57 \pm 0.17$ & $99.29 \pm 0.17$   \\
FC100& 5 way 5 shot      & $\textbf{85.25} \pm 0.01$ & $78.14 \pm 1.97$ & $74.19 \pm 1.35$ & $84.75 \pm 0.65$  \\
FC100& 10 way 5 shot       & $\textbf{86.93} \pm 0.89$ & $67.18 \pm 0.75$ & $69.10 \pm 0.85$ & $81.18 \pm 0.39$ \\
\bottomrule
\end{tabular}
\end{table*}
Using  all the assumptions and lemmas, we can obtain the following theorem (For ease of reading, some parameter Settings are omitted here. The specific parameters can be found in Appendix K.2):
\begin{theorem}\label{theorem:convergence_rate}
Under Assumptions \ref{assump:upper_level}, \ref{assump:lower_level} and Lemma \ref{lemma:bound_of_nablaP},
with $\frac{1}{\mu_g}(1-\frac{1}{4(2\pi)^{1/4}\sqrt{d_2}L_p})\leq\eta<\frac{1}{\mu_g}$, $Q=\frac{1}{\mu_g\eta}\ln\frac{C_{gxy}C_{fy}K}{\mu_g}$, $0\leq a\leq 2$, $\alpha=c_1\eta_k$, $\beta=c_2\eta_k$, $L_0=\max(L_1(\frac{d_2}{\delta}),L_2(\frac{d_2}{\delta}))>1$, $\Phi_1
        =\mathbb{E}[F_{\delta}(x_{1})+\frac{10L_0^a c_l}{\tau\mu_g c_u}\|y_{1}-y^*(x_{1})\|^2 + c_l(\|w_{1}-\bar{\nabla} f_{\delta}(x_{1},y_{1})-R_{1}\|^2+\|\nabla_y g(x_{1},y_{1})-v_{1}\|^2)]$, and $\eta_k=\frac{t}{(m+k)^{1/2}}$, $t>0$, we have
\begin{align}
        &\min\{\|h\|:h\in\bar{\partial}_{\delta} F(x_r)\}
        \leq\frac{4m^{1/4}\sqrt{G}}{\sqrt{Kt}}+\frac{4\sqrt{G}}{(Kt)^{1/4}}.\nonumber
    \end{align}
    where $G=\frac{\Phi_1-\Phi^*}{\gamma c_l}+\frac{17t}{4K^2}(m+K)^{1/2}+\frac{4}{3tK^2}(m+K)^{3/2}+(m\sigma_f(d_2))t^2\ln (m+K)$; the range of $c_1$,$c_2$ $\gamma$, $\tau$ and $m$ are given  in the appendix. 
\end{theorem}

\begin{remark}
    Let $t=\mathcal{O}(1)$, $m=\mathcal{O}(1)$, $\sigma_f(d_2)=\mathcal{O}(d_2)$ and $
    \ln(m+K)=\tilde{\mathcal{O}}(1)$, we have $\sqrt{G}=\tilde{\mathcal{O}}((\frac{d_2}{\delta})^{a/2}+\sqrt{d_2})$. Thus, our proposed DMLCBO can converge to a $(\delta, \epsilon)$ stationary point at the rate of $\tilde{\mathcal{O}}(((\frac{d_2}{\delta})^{a/2}+\sqrt{d_2})K^{-1/4})$ with properly choosing the hyper-parameters. Then, setting $a=0$ and $\min\{\|h\|:h\in\bar{\partial}_{\delta} F(x_r)\}\leq\tilde{\mathcal{O}}(((\frac{d_2}{\delta})^{a/2}+\sqrt{d_2})K^{-1/4})\leq \epsilon$, we have $K=\tilde{\mathcal{O}}(d_2^2\epsilon^{-4})$. Obviously, by setting $a=0$, our method can avoid the influence of $\delta$ on the convergence iterations \cite{lin2022gradient} and obtain the same convergence iteration number as the traditional bilevel optimization method without the Lipschitz assumption on the stochastic gradient estimation \cite{huang2021biadam}. Then, since we need to approximate $c(Q)\times d_2$ Jacobian-vector products and $
    \mathcal{O}(p)$ time to calculate the projection in each iteration, the average computational complexity of approximating hypergradient is $\mathcal{O}(Qd_2p+d_1)$. Therefore, the total computational complexity of our method is $\mathcal{\tilde{O}}((Qd_2p+d_1)d_2^2\epsilon^{-4})$. It is worth noting that since we approximate $d_2$ Jacobians in parallel, the overall execution time of our algorithm can be significantly reduced.

\end{remark}

\section{Experiments}
In this section, we compare the performance of our method with SOTA methods for LCBO in two applications. (Detailed settings are given in our Appendix.)
\subsection{Baselines}
We compare our method with the following  LCBO methods.
\begin{enumerate}
    \item \textbf{V-PBGD}. The method proposed in \cite{shen2023penalty} uses the value function method to solve the bilevel optimization problem. 
    \item \textbf{RMD-PCD}. The method proposed in \cite{bertrand2022implicit} which uses the reverse method to calculate the hyper-gradient.
    \item \textbf{Approx}. The method proposed in \cite{pedregosa2016hyperparameter} solves a linear optimization problem to calculate the hypergradient.
\end{enumerate}
We implement all the methods by Pytorch \cite{paszke2019pytorch}. Since JaxOpt \cite{blondel2022efficient} is implemented by JAX \cite{jax2018github}, for a fair comparison, we use Approx with the Jacobian calculating methods in \cite{martins2016softmax,djolonga2017differentiable,blondel2020fast,niculae2017regularized,vaiter2013local,cherkaoui2020learning} as a replacement of JaxOpt, which uses the same method to calculate the hypergradient. We run all the methods 10 times on a PC with four 1080Ti GPUs.

\subsection{Applications}

\noindent \textbf{Data hyper-cleaning.} In this experiment, we evaluate the performance of all the methods in the application named data hyper-cleaning.
In many real-world applications, the training set and testing set may have different distributions. To reduce the discrepancy between the two distributions, each data point will be given an additional importance weight, which is called data hyper-clean.
This problem can be formulated as 
\begin{align}
	&\min_{x\in\mathbb{R}^{d_1}} \sum_{\mathcal{D}_{val}}\ell\left(y^*(x)^{\top}a_i,b_i\right)\nonumber\\
	s.t.&\quad y^*(x)=\mathop{\arg\min}_{\|y\|_1\leq r}\sum_{\mathcal{D}_{tr}}[\sigma(x)]_i\ell\left(y^{\top}a_i,b_i\right)+c\|y\|^2, \nonumber 
\end{align} 
where $r>0$, $\mathcal{D}_{tr}$ and $\mathcal{D}_{val}$ denote the training set and validation set respectively; $(a_i, b_i)$ denotes the data point; $\sigma(\cdot)\coloneqq1/(1+exp(-\cdot))$ is the Sigmoid function; $\ell(\cdot,\cdot)$ is the loss function; $c>0$ is the regularization parameter used to ensure the lower-level problem to be strongly convex. In this experiment, an additional $\ell_1$ is added to the lower-level problem to ensure the sparsity of the model.

\noindent\textbf{Meta-learning.} Meta-learning for few-shot learning is to learn a shared prior parameter across a distribution of tasks, such that a simple learning step with few-shot data based on the prior leads to a good adaptation to the task in the distribution. In particular, the training task $\mathcal{T}_i$ is sampled from distribution $P_{\mathcal{T}}$. Each task $\mathcal{T}_i$ is characterized by its training data $\mathcal{D}_{tr}^i$ and the test data $\mathcal{D}_{te}^i$. The upper level is to extract features from input data and multi-class SVM served as the base learner in the lower-level optimization to classify the data on its extracted features. This problem can be formulated as 
\begin{align}
    &\min_{\phi}\sum_{\mathcal{T}_i\sim P_{\mathcal{T}}}\mathcal{L}(y_i^*(x),\phi_x,\mathcal{D}_{te}^i)\\
    s.t.&\quad y_i^*(x)=\mathop{\arg\min}_{0\leq y_i\leq Ce} \frac{1}{2}y_i^{\top}Gy_i-e^{\top}y_i \nonumber
\end{align}
where $\phi_x$ denotes the network parameterized by $x$; $e$ denotes the vector with all elements equal to 1; $G=K\odot(VV^{\top})$, $K_{ij}=\phi(a_i)^{\top}\phi(a_j)$, $V_{ij}=\frac{1}{k-1}$ if $b_i\not=j$ otherwise $V_{ij}=1$, $(a_j,b_j)\in\mathcal{D}_{tr}^i$, and $K$ denotes the number of classes.


\subsection{Results}
We have presented the test accuracy results for all methods in Tables \ref{table:data_reweight} and \ref{table:poison_attack_acc}, and visualized the testing performance as a function of training time in Figures \ref{fig:data_reweight_time_vs_acc} and \ref{fig:poison_attack_time_vs_acc}. Upon closer examination of the data presented in Tables \ref{table:data_reweight} and \ref{table:poison_attack_acc}, we can find that our method achieves similar performances in some cases and sometimes gets results better than other methods. One possible reason is that RMD-PCD, Approx, and V-PBGD highly depend on the solutions to the lower-level problem. Once the approximated solution is not good enough then they may obtain bad performance. 
From all these results, we can conclude that our DMLCBO
outperforms other methods in terms of effectiveness.

\section{Conclusion}
In this paper, without using the restive assumptions on $y^*(x)$, we propose a new method to derive the hypergradient for LCBO. Then, we leverage randomized smoothing to approximate the hypergradient. Then, using our new hypergradient approximation, we propose a \textit{single-loop single-timescale} algorithm based on the double-momentum method and adaptive step size method, which updates the lower- and upper-level variables simultaneously. Theoretically, we prove our methods can converge to the $(\delta,\epsilon)$-stationary point with $\tilde{\mathcal{O}}(d_2^2\epsilon^{-4})$. The experimental results in data hypercleaning and meta-learning demonstrate the effectiveness of our proposed method. 

\section*{Acknowledgements}
  Bin Gu was supported by the Natural Science Foundation of China under Grant No.62076138. Yi Chang was supported by the Natural Science Foundation of China under Grant No.U2341229 and the National Key R\&D Program of China under Grant No.2023YFF0905400

\section*{Impact Statement}
This paper presents work on the theoretical analysis of the bilevel optimization problem. There are many potential societal consequences of our work, none of which we feel must be specifically highlighted here.
\bibliography{example_paper}
\bibliographystyle{icml2024}

\newpage
\appendix
\onecolumn

\section*{Appendix }

\section{Experimental Setting}
\noindent \textbf{Data hyper-cleaning.} In this experiment, we evaluate all the methods on the datasets MNIST, FashionMNIST, CodRNA, and Madelon \footnote{https://www.csie.ntu.edu.tw/~cjlin/libsvmtools/datasets/}. For MNIST and FashionMNIST, we choose two classes to conduct a binary classification. In addition, we flip $30\%$ of the labels in the training set as the noisy data. We set $r=1$ for all the datasets.  For our method, we search the step size from the set $\{1,10^{-1},10^{-2},10^{-3},10^{-4},10^{-5}\}$.  For other methods, we search the step size from the set $\{10,1,10^{-1},10^{-2},10^{-3},10^{-4},10^{-5}\}$. Following the default setting in \cite{ji2021bilevel}, we set $Q=3$ and $\eta=0.5$ for our method. In addition, we set $\eta_k={1}/{\sqrt{100+k}}$, $c_1=10$ and $c_2=10$ for our method. For V-PBGD, RMD-PCD, and Approx, following the setting in \cite{pedregosa2016hyperparameter}, we set the inner iteration number at $100$. We run all the methods for $10000$ iterations and evaluate the test accuracy for every $100$ iteration. We set $\delta=1e-6$.

\noindent\textbf{Meta learning.} In this experiment, we evaluate all the methods on the datasets Omniglot and FC100. We set $C=10$. We use a network with four convolution-ELU-Maxpooling layers to extract the features that output a $\mathbb{R}^{100}$ vector as the feature.  For our method, we search the step size from the set $\{1,10^{-1},10^{-2},10^{-3},10^{-4},10^{-5}\}$.  For other methods, we search the step size from the set $\{10,1,10^{-1},10^{-2},10^{-3},10^{-4},10^{-5}\}$. Following the default setting in \cite{ji2021bilevel}, we set $Q=3$ and $\eta=0.5$ for our method. In addition, we set $\eta_k={1}/{\sqrt{100+k}}$, $c_1=10$ and $c_2=10$ for our method. For V-PBGD, RMD-PCD, and Approx, we also set the inner iteration number at $3$. We set $\delta=1e-6$.

\section{Ablation Study}

In this section, we conduct ablation experiments on the hyper-parameters $Q$ and $\eta$. To control the variables, we explore the effect of each hyper-parameter while keeping the other hyper-parameters as default as shown in the experimental setups. We search the step size of both $x$ and $y$. We present the results in Tables \ref{table:data_reweight_Q}, \ref{table:data_reweight_eta}. We can find that increasing $Q$ will lead to a long training time. In addition, setting $Q=1$ usually leads to the worst results, which is because setting $Q=1$ means ignoring the inverse of the Hessian matrix in our hypergradient and may not converge to our stationary point.

\begin{table}[h]
\caption{Test accuracy (\%) of our method with different Q in data hyper-cleaning. (Higher is better.)}
\label{table:data_reweight_Q}
\centering
\begin{tabular}{lcccc}
\toprule
Datasets        & Q=1           &  Q=3       & Q=5           &  Q=7             \\ \hline
CodRNA     & $80.25 \pm 0.07$   & $80.42 \pm 0.05$ & $80.55 \pm 0.03$ & $\textbf{80.56} \pm 0.04$ \\
MNIST 6vs9      & $93.65 \pm 0.01$ & $94.24\pm 0.01$ & $94.15 \pm 0.01$ & $\textbf{94.35} \pm 0.01$ \\
Madelon    & $55.23 \pm 1.22$ & $56.33 \pm 1.83$ & $\textbf{56.67} \pm 1.24$ & $56.23 \pm 1.56$\\
FashionMNIST 1vs7  & $88.25 \pm 0.56$ & $\textbf{88.85} \pm 0.60$ & $88.75 \pm 0.40$  & $88.55 \pm 0.33$ \\
\bottomrule
\end{tabular}
\end{table}

\begin{table}[h]
\caption{Test accuracy (\%) of our method with different $\eta$ in data hyper-cleaning. (Higher is better.)}
\label{table:data_reweight_eta}
\centering
\begin{tabular}{lccc}
\toprule
Datasets        & $\eta=0.1$           &  $\eta=0.5$       & $\eta=1$           \\ \hline
CodRNA     & $80.45 \pm 0.04$   & $80.42 \pm 0.05$ & $80.33 \pm 0.06$  \\
MNIST 6vs9      & $94.65 \pm 0.01$ & $94.24\pm 0.01$ & $94.35 \pm 0.01$  \\
Madelon    & $56.23 \pm 1.14$ & $56.33 \pm 1.83$ & $56.23 \pm 1.56$\\
FashionMNIST 1vs7  & $88.65 \pm 0.67$ & $88.85\pm 0.60$ & $88.67 \pm 0.23$   \\
\bottomrule
\end{tabular}
\end{table}



\section{Proof of Lemma \ref{lemma:Lip_y_star}}
\begin{proof}
    Here, we follow the proof in \cite{van2021variable}. For given $x_1$ and $x_2$, we have the corresponding unique optimal solutions $y^*(x_1)$ and $y^*(x_2)$. For the constrained lower-level problem, we have the following optimal conditions
	\begin{align}
		0\in \nabla_{y}g(x_1,y^*(x_1))+\partial\delta_{\mathcal{Y}}(y^*(x_1)),\quad 0\in \nabla_{y}g(x_2,y^*(x_2))+\partial\delta_{\mathcal{Y}}(y^*(x_2))
	\end{align}
	where $\delta_{\mathcal{Y}}(\cdot)$ is the indicator function of the constriant.
	
	Since $g$ is strongly convex for any given $\tilde{x}$ and $\mathcal{Y}$ is convex, we have 
	\begin{align}
		\mu_g\|y^*(x_1)-y^*(x_2)\|^2\leq\left\langle \left(\nabla_{y}g(\tilde{x},y^*(x_1))+h_1\right) -\left(\nabla_{y}g(\tilde{x},y^*(x_2))+h_2 \right), y^*(x_1)-y^*(x_2)  \right\rangle
	\end{align}
	where $h_1\in \partial\delta_{\mathcal{Y}}(y^*(x_1))$ and $h_2\in \partial\delta_{\mathcal{Y}}(y^*(x_2))$. Make the particular choices $h_1=-\nabla_{y}g(x_1,y^*(x_1))\in \partial\delta_{\mathcal{Y}}(y^*(x_1))$ and $h_2=-\nabla_{y}g(x_2,y^*(x_2))\in \partial\delta_{\mathcal{Y}}(y^*(x_2))$, we have 
	\begin{align}
		&\mu_g\|y^*(x_1)-y^*(x_2)\|^2\nonumber\\
		\leq&\left\langle \nabla_{y}g(\tilde{x},y^*(x_1))-\nabla_{y}g(x_1,y^*(x_1)), y^*(x_1)-y^*(x_2)  \right\rangle\nonumber\\ &+\left\langle\nabla_{y}g(x_2,y^*(x_2))-\nabla_{y}g(\tilde{x},y^*(x_2)) , y^*(x_1)-y^*(x_2)  \right\rangle
	\end{align}
	Then, setting $\tilde{x}=x_1$, we have 
	\begin{align}
		&\mu_g\|y^*(x_1)-y^*(x_2)\|^2\nonumber\\
		\leq&\left\langle\nabla_{y}g(x_2,y^*(x_2))-\nabla_{y}g(x_1,y^*(x_2)) , y^*(x_1)-y^*(x_2)  \right\rangle\nonumber\\
		\leq&\|\nabla_{y}g(x_2,y^*(x_2))-\nabla_{y}g(x_1,y^*(x_2))\|\| y^*(x_1)-y^*(x_2)  \|\nonumber\\
		\leq&L_{g}\|x_1-x_2\|\| y^*(x_1)-y^*(x_2)  \|
	\end{align}
	where the last inequality is due to Assumption \ref{assump:lower_level}. Rearrange above inequality, we have 
	\begin{align}
		\|y^*(x_1)-y^*(x_2)\|
		\leq \frac{L_{g}}{\mu_g}\|x_1-x_2\|
	\end{align}
	That completes the proof.
\end{proof}

\section{Proof of Lemma \ref{lemma:converge_p_delta}}
\begin{proof}
    First, we show that the limit exists. By Lipschitzness of the projection operator and Jenson inequality, we know that $\partial_{\delta} \mathcal{P}_{\mathcal{Y}}(z)$ lies in a bounded ball with radius 1. For any sequence of $\delta_k$ with $\delta_k\downarrow 0$, we know that $\partial_{\delta_{k+1}} \mathcal{P}_{\mathcal{Y}}(z)\subset \partial_{\delta_k} \mathcal{P}_{\mathcal{Y}}(z)$. Therefore, the limit exists by the monotone convergence theorem.

    Next, we show that $\lim_{\delta\downarrow0}\partial_{\delta} \mathcal{P}_{\mathcal{Y}}(z)=\partial \mathcal{P}_{\mathcal{Y}}(z)$. According to the Proposition 2.6.2 in \cite{clarke1990optimization}, we have 
    \begin{align}
        \partial \mathcal{P}_{\mathcal{Y}}(z)=\bigcap_{\delta>0}\bigcup_{z'\in\mathbb{B}_{\delta}(z)}\partial \mathcal{P}_{\mathcal{Y}}(z').
    \end{align}
    Then, using the fact 
    \begin{align}
        \bigcup_{z'\in\mathbb{B}_{\delta}(z)}\partial \mathcal{P}_{\mathcal{Y}}(z')\subseteq co\left(\bigcup_{z'\in\mathbb{B}_{\delta}(z)}\partial \mathcal{P}_{\mathcal{Y}}(z')\right)=\partial_{\delta} \mathcal{P}_{\mathcal{Y}}(z),
    \end{align}
    we have $\partial \mathcal{P}_{\mathcal{Y}}(z)\subseteq \lim_{\delta\downarrow0}\partial_{\delta} \mathcal{P}_{\mathcal{Y}}(z)$.

    Using the upper semicontinuous of $\partial \mathcal{P}_{\mathcal{Y}}(z)$ at $z$ \cite{clarke1990optimization}, we have that for any $\epsilon>0$, there exists $\delta>0$ such that 
    \begin{align}
        \bigcup_{z'\in\mathbb{B}_{\delta}(z)}\partial \mathcal{P}_{\mathcal{Y}}(z')\subseteq \partial \mathcal{P}_{\mathcal{Y}}(z)+\epsilon B
    \end{align}
    Then, by convexity of $\partial \mathcal{P}_{\mathcal{Y}}(z)$ and $\epsilon B$, we know that their Minknowski sum $\partial \mathcal{P}_{\mathcal{Y}}(z)+\epsilon B$ isconvex. Therefore, we conclude that for any $\epsilon>0$, there exists $\delta>0$ such that 
    \begin{align}
        \partial_{\delta} \mathcal{P}_{\mathcal{Y}}(z)=co\left(\bigcup_{z'\in\mathbb{B}_{\delta}(z)}\partial \mathcal{P}_{\mathcal{Y}}(z')\right)\subseteq \partial \mathcal{P}_{\mathcal{Y}}(z)+\epsilon B
    \end{align}
    Therefore, we have $ \lim_{\delta\downarrow0}\partial_{\delta} \mathcal{P}_{\mathcal{Y}}(z)\subseteq\partial \mathcal{P}_{\mathcal{Y}}(z)$. Finally, we obtain $\lim_{\delta\downarrow0}\partial_{\delta} \mathcal{P}_{\mathcal{Y}}(z)=\partial \mathcal{P}_{\mathcal{Y}}(z)$.
\end{proof}

\section{Proof of Proposition \ref{prop1}}
\begin{proof}
    Let $u\in\mathbb{R}^d$ denote a random variable distributed uniformly on $\mathbb{B}_1(0)$. Since $\mathcal{P}_{\mathcal{Y}}(\cdot)$ is 1-Lipschitz, we have
    \begin{align}
        \|\mathcal{P}_{\mathcal{Y\delta}}(z)-\mathcal{P}_{\mathcal{Y}}(z)\|=\left\|\mathbb{E}[\mathcal{P}_{\mathcal{Y}}(z+\delta u)-\mathcal{P}_{\mathcal{Y}}(z)]\right\|\leq \delta \cdot 1\cdot\mathbb{E}[\|u\|]=\delta
    \end{align}
    Since $\mathcal{P}_{\mathcal{Y}}(z)$ is 1-Lipschitz, we have 
    \begin{align}
        &\|\mathcal{P}_{\mathcal{Y\delta}}(z)-\mathcal{P}_{\mathcal{Y\delta}}(z')\|
        =\left\|\mathbb{E}[\mathcal{P}_{\mathcal{Y}}(z+\delta u)-\mathcal{P}_{\mathcal{Y}}(z'+\delta u)]\right\|\leq \|\mathbb{E}[\|z-z'\|]\|=\|z-z'\|
    \end{align}

    For the Jacobian matrix, we have 
     \begin{align}
         \nabla\mathcal{P}_{\mathcal{Y\delta}}(z)=\begin{bmatrix}
             (\nabla\mathcal{P}_{\mathcal{Y\delta}}^1(z))^{\top}\\
             \vdots\\
             (\nabla\mathcal{P}_{\mathcal{Y\delta}}^{d_2}(z))^{\top}
         \end{bmatrix}
     \end{align}
     For each element $\nabla\mathcal{P}_{\mathcal{Y\delta}}^{i}(z)$, $(i\in[1,\cdots,d_2])$, according to Theorem 3.1 in \cite{lin2022gradient}, we have $\nabla\mathcal{P}_{\mathcal{Y\delta}}^{i}(z)=\mathbb{E}[\nabla\mathcal{P}_{\mathcal{Y}}^i(z+\delta u)]$. Therefore, we can easily obtain $\nabla\mathcal{P}_{\mathcal{Y\delta}}(z)=\mathbb{E}[\nabla\mathcal{P}_{\mathcal{Y}}(z+\delta u)]$. Then, according to \cite{lin2022gradient}, we have 
     \begin{align}
        &\|\nabla\mathcal{P}_{\mathcal{Y\delta}}(z)-\nabla\mathcal{P}_{\mathcal{Y\delta}}(z')\|
         \leq\|\nabla\mathcal{P}_{\mathcal{Y\delta}}(z)-\nabla\mathcal{P}_{\mathcal{Y\delta}}(z')\|_F
         =\left(\sum_{i=1}^{d_2} \|\nabla\mathcal{P}_{\mathcal{Y\delta}}^i(z)-\nabla\mathcal{P}_{\mathcal{Y\delta}}^i(z')\|^2 \right)^{1/2}\nonumber\\
         \leq&\left(\sum_{i=1}^{d_2} (\frac{cL_p\sqrt{d_2}}{\delta}\|z-z'\|)^2 \right)^{1/2}=\frac{cd_2L_p}{\delta}\|z-z'\|
     \end{align}
     where $c>0$ is a constant.
\end{proof}
\section{Proof of Proposition \ref{prop2}}
\begin{proof}
    According to Theorem 3.1 in \cite{lin2022gradient}, for each element $\nabla\mathcal{P}_{\mathcal{Y}\delta}^i(z)$ in $\nabla\mathcal{P}_{\mathcal{Y}\delta}(z)$, we have $\nabla\mathcal{P}_{\mathcal{Y}\delta}^i(z)\in \partial_{\delta} \mathcal{P}_{\mathcal{Y}}^i(z)$. Therefore, we have $\nabla\mathcal{P}_{\mathcal{Y}\delta}(z)\in \partial_{\delta} \mathcal{P}_{\mathcal{Y}}(z)$
\end{proof}

\section{Lipschitz Continuousness of $\nabla F_{\delta}(x)$}
Here, we prove $\nabla F_{\delta}(x)$ is Lipschitz continuous. We first give several useful lemmas.

\begin{lemma}\label{lemma:lip_f_y}
    (\textbf{Lipschitz continuous of the approximation of hypergradient on $x$ and $y$}.) Under Assumptions \ref{assump:upper_level}, \ref{assump:lower_level},  $\nabla f_{\delta}(x,y)$ is Lipschitz continuous on $y\in\mathcal{Y}$ and $x\in\mathcal{X}$, respectively, such that we have 
    \begin{align}
        &\|\nabla f_{\delta}(x,y_1)-\nabla f_{\delta}(x,y_2)\|\leq L_1(\frac{d_2}{\delta})\|y_1-y_2\|\\
        &\|\nabla f_{\delta}(x_1,y)-\nabla f_{\delta}(x_2,y)\|\leq L_2(\frac{d_2}{\delta})\|x_1-x_2\|
    \end{align}
    where $L_1(\frac{d_2}{\delta})=L_f+ \frac{L_{gxy}C_{fy}}{\mu_g}+ \frac{C_{gxy}C_{fy}}{\mu_g}(\frac{cd_2}{\delta}+\eta \frac{cd_2}{\delta}L_{g})+  C_{gxy}C_{fy}\frac{1}{\eta\mu_g^2}(\frac{cd_2}{\delta}+\eta \frac{cd_2}{\delta}L_{g})(1-\eta\mu_g)+\frac{L_{gyy}}{\mu_g^2}+ \frac{C_{gxy}}{\mu_g}L_{f}$ and $L_2(\frac{d_2}{\delta})=L_f+ \frac{L_{gxy}C_{fy}}{\mu_g}+ \frac{C_{gxy}C_{fy}}{\mu_g}\eta\frac{cd_2}{\delta}L_{g}+  C_{gxy}C_{fy}\left(\eta\frac{1}{\mu_g^2}\frac{cd_2}{\delta}L_{g}(1-\eta\mu_g)+\frac{L_{gyy}}{\mu_g^2}\right)+ \frac{C_{gxy}}{\mu_g}L_{f}$
\end{lemma}
\begin{proof}
    Using the definition of $\nabla f_{\delta}(x,y)$, $z_1=y_1-\nabla_yg(x,y_1)$ and $z_2=y_2-\nabla_yg(x,y_2)$, we have
    \begin{align}
        &\|\nabla f_{\delta}(x,y_1)-\nabla f_{\delta}(x,y_2)\|\nonumber\\
        =&\|\nabla_x f(x,y_1)- \eta \nabla_{xy}^2g(x,y_1) \nabla\mathcal{P}_{\mathcal{Y}\delta}(z_1)^{\top}\left[(I_{d_2}-(I_{d_2}-\eta\nabla_{yy}^2g(x,y_1))\nabla\mathcal{P}_{\mathcal{Y}\delta}(z_1)^{\top}\right]^{-1}\nabla_yf(x,y_1)\nonumber\\
        &-\nabla_x f(x,y_2)+ \eta \nabla_{xy}^2g(x,y_2) \nabla\mathcal{P}_{\mathcal{Y}\delta}(z_2)^{\top}\left[(I_{d_2}-(I_{d_2}-\eta\nabla_{yy}^2g(x,y_2))\nabla\mathcal{P}_{\mathcal{Y}\delta}(z_2)^{\top}\right]^{-1}\nabla_yf(x,y_2)\|\nonumber\\
        \leq&\|\nabla_xf(x,y_1)-\nabla f(x,y_2)\|\nonumber\\
        &+\eta\|\nabla_{xy}^2g(x,y_2)-\nabla_{xy}^2g(x,y_1)\| \|\nabla\mathcal{P}_{\mathcal{Y}\delta}(z_2)\|\|\left[I_{d_2}-(I_{d_2}-\eta\nabla_{yy}^2g(x,y_2))\nabla\mathcal{P}_{\mathcal{Y}\delta}(z_2)^{\top}\right]^{-1}\|\|\nabla_yf(x,y_2)\|\nonumber\\
        &+\eta \|\nabla_{xy}^2g(x,y_1)\| \|\nabla\mathcal{P}_{\mathcal{Y}\delta}(z_2)-\nabla\mathcal{P}_{\mathcal{Y}\delta}(z_1)\|\|\left[I_{d_2}-(I_{d_2}-\eta\nabla_{yy}^2g(x,y_2))\nabla\mathcal{P}_{\mathcal{Y}\delta}(z_2)^{\top}\right]^{-1}\|\|\nabla_yf(x,y_2)\|\nonumber\\
        &+\eta \|\nabla_{xy}^2g(x,y_1)\| \|\nabla\mathcal{P}_{\mathcal{Y}\delta}(z_1)\|\|\left[I_{d_2}-(I_{d_2}-\eta\nabla_{yy}^2g(x,y_2))\nabla\mathcal{P}_{\mathcal{Y}\delta}(z_2)^{\top}\right]^{-1}\nonumber\\
        &-\left[I_{d_2}-(I_{d_2}-\eta\nabla_{yy}^2g(x,y_1))\nabla\mathcal{P}_{\mathcal{Y}\delta}(z_1)^{\top}\right]^{-1}\|\|\nabla_yf(x,y_2)\|\nonumber\\
        &+\eta \|\nabla_{xy}^2g(x,y_1)\| \|\nabla\mathcal{P}_{\mathcal{Y}\delta}(z_1)\|\|\left[I_{d_2}-(I_{d_2}-\eta\nabla_{yy}^2g(x,y_1))\nabla\mathcal{P}_{\mathcal{Y}\delta}(z_1)^{\top}\right]^{-1}\|\|\nabla_yf(x,y_2)-\nabla_yf(x,y_1)\|
    \end{align}
    We have 
    \begin{align}
        &\|\nabla_xf(x,y_1)-\nabla_x f(x,y_2)\|\leq L_f\|y_1-y_2\|,\\ 
        &\|\nabla_yf(x,y_1)-\nabla_y f(x,y_2)\|\leq L_f\|y_1-y_2\|, \\
        &\|\nabla_{xy}^2g(x,y_2)-\nabla_{xy}^2g(x,y_1)\|\leq L_{gxy}\|y_2-y_1\|,\\
        &\|\nabla \mathcal{P}_{\mathcal{Y}\delta}(z_1)\|\leq1, \quad\|\nabla f(x,y_2)\|\leq C_{fy},\quad\|\nabla_{xy}^2g(x,y_1)\|\leq C_{gxy}
    \end{align}

     Since $\|(I_{d_2}-\eta\nabla_{yy}^2g(x,y_2))\nabla\mathcal{P}_{\mathcal{Y}\delta}(z_2)^{\top}\|\leq\|\nabla\mathcal{P}_{\mathcal{Y}\delta}(z_2)\|\|I_{d_2}-\eta\nabla_{yy}^2g(x,y_2)\|\leq1-\eta\mu_g\leq1$ we have 
    \begin{align}
        &\|\left[I_{d_2}-(I_{d_2}-\eta\nabla_{yy}^2g(x,y_2))\nabla\mathcal{P}_{\mathcal{Y}\delta}(z_2)^{\top}\right]^{-1}\|\nonumber\\
        \leq&\frac{1}{1-\|(I_{d_2}-\eta\nabla_{yy}^2g(x,y_2))\nabla\mathcal{P}_{\mathcal{Y}\delta}(z_2)^{\top}\|}\nonumber\\
        \leq&\frac{1}{\eta\mu_g}
    \end{align}
    \begin{align}
        &\|\nabla\mathcal{P}_{\mathcal{Y}\delta}(z_2)-\nabla\mathcal{P}_{\mathcal{Y}\delta}(z_1)\|\nonumber\\
        \leq&\frac{cd_2}{\delta}\|z_2-z_1\|\nonumber\\
        =&\frac{cd_2}{\delta}\|y_2-\eta\nabla_yg(x,y_2)-y_1+\eta\nabla_yg(x,y_1)\|\nonumber\\
        \leq&\frac{cd_2}{\delta}\|y_2-y_1\|+\eta \frac{cd_2}{\delta}\|\nabla_yg(x,y_1)-\eta\nabla_yg(x,y_2)\|\nonumber\\
        \leq&(\frac{cd_2}{\delta}+\eta \frac{cd_2}{\delta}L_{g})\|y_2-y_1\|
    \end{align}
    Using the inequality $\|H_2^{-1}-H_1^{-1}\|\leq\|H_1^{-1}(H_1-H_2)H_2^{-1}\|\leq\|H_1^{-1}\|\|H_1-H_2\|\|H_2^{-1}\|$, we have
    \begin{align}
        &\|\left[I_{d_2}-(I_{d_2}-\eta\nabla_{yy}^2g(x,y_2))\nabla\mathcal{P}_{\mathcal{Y}\delta}(z_2)^{\top}\right]^{-1}-\left[I_{d_2}-(I_{d_2}-\eta\nabla_{yy}^2g(x,y_1))\nabla\mathcal{P}_{\mathcal{Y}\delta}(z_1)^{\top}\right]^{-1}\|\nonumber\\
        \leq&\frac{1}{\eta^2\mu_g^2}\|(I_{d_2}-\eta\nabla_{yy}^2g(x,y_2))\nabla\mathcal{P}_{\mathcal{Y}\delta}(z_2)^{\top}-(I_{d_2}-\eta\nabla_{yy}^2g(x,y_1))\nabla\mathcal{P}_{\mathcal{Y}\delta}(z_1)^{\top}\|\nonumber\\
        \leq&\frac{1}{\eta^2\mu_g^2}\|\nabla\mathcal{P}_{\mathcal{Y}\delta}(z_2)-\nabla\mathcal{P}_{\mathcal{Y}\delta}(z_1)\|\|I_{d_2}-\eta\nabla_{yy}^2g(x,y_2)\|\nonumber\\
        &+\frac{1}{\eta^2\mu_g^2}\|\nabla\mathcal{P}_{\mathcal{Y}\delta}(z_1)\|\|I_{d_2}-\eta\nabla_{yy}^2g(x,y_2)-I_{d_2}+\eta\nabla_{yy}^2g(x,y_1)\|\nonumber\\
        \leq&\left(\frac{1}{\eta^2\mu_g^2}(\frac{cd_2}{\delta}+\eta \frac{cd_2}{\delta}L_{g})(1-\eta\mu_g)+\frac{L_{gyy}}{\eta\mu_g^2}\right)\|y_2-y_1\|
    \end{align}
    Therefore, using the above inequalities, we obtain
    \begin{align}
        &\|\nabla f_{\delta}(x,y_1)-\nabla f_{\delta}(x,y_2\nabla f_{\delta}(x,y))\|\nonumber\\
        \leq&L_f+ \frac{L_{gxy}C_{fy}}{\mu_g}+ \frac{C_{gxy}C_{fy}}{\mu_g}(\frac{cd_2}{\delta}+\eta \frac{cd_2}{\delta}L_{g})\nonumber\\
        &+  C_{gxy}C_{fy}\left(\frac{1}{\eta\mu_g^2}(\frac{cd_2}{\delta}+\eta \frac{cd_2}{\delta}L_{g})(1-\eta\mu_g)+\frac{L_{gyy}}{\mu_g^2}\right)+ \frac{C_{gxy}}{\mu_g}L_{f}\|y_2-y_1\|
    \end{align}

    For the second statement, let $z_1=y-\nabla_yg(x_1,y)$ and $z_2=y-\nabla_yg(x_2,y)$, we have
    \begin{align}
        &\|\nabla f_{\delta}(x_1,y)-\nabla f_{\delta}(x_2,y)\|\nonumber\\
        =&\|\nabla_x f(x_1,y)- \eta \nabla_{xy}^2g(x_1,y) \nabla\mathcal{P}_{\mathcal{Y}\delta}(z_1)^{\top}\left[I_{d_2}-(I_{d_2}-\eta\nabla_{yy}^2g(x_1,y))\nabla\mathcal{P}_{\mathcal{Y}\delta}(z_1)^{\top}\right]^{-1}\nabla_yf(x_1,y)\nonumber\\
        &-\nabla_x f(x_2,y)+ \eta \nabla_{xy}^2g(x_2,y) \nabla\mathcal{P}_{\mathcal{Y}\delta}(z_2)^{\top}\left[I_{d_2}-(I_{d_2}-\eta\nabla_{yy}^2g(x_2,y))\nabla\mathcal{P}_{\mathcal{Y}\delta}(z_2)^{\top}\right]^{-1}\nabla_yf(x_2,y)\|\nonumber\\
        \leq&\|\nabla_xf(x_1,y)-\nabla f(x_2,y)\|\nonumber\\
        &+\eta\|\nabla_{xy}^2g(x_2,y)-\nabla_{xy}^2g(x_1,y)\| \|\nabla\mathcal{P}_{\mathcal{Y}\delta}(z_2)\|\|\left[I_{d_2}-(I_{d_2}-\eta\nabla_{yy}^2g(x_2,y))\nabla\mathcal{P}_{\mathcal{Y}\delta}(z_2)^{\top}\right]^{-1}\|\|\nabla_yf(x_2,y)\|\nonumber\\
        &+\eta \|\nabla_{xy}^2g(x_1,y)\| \|\nabla\mathcal{P}_{\mathcal{Y}\delta}(z_2)-\nabla\mathcal{P}_{\mathcal{Y}\delta}(z_1)\|\|\left[I_{d_2}-(I_{d_2}-\eta\nabla_{yy}^2g(x_2,y))\nabla\mathcal{P}_{\mathcal{Y}\delta}(z_2)^{\top}\right]^{-1}\|\|\nabla_yf(x_2,y)\|\nonumber\\
        &+\eta \|\nabla_{xy}^2g(x_1,y)\| \|\nabla\mathcal{P}_{\mathcal{Y}\delta}(z_1)\|\|\left[I_{d_2}-(I_{d_2}-\eta\nabla_{yy}^2g(x_2,y))\nabla\mathcal{P}_{\mathcal{Y}\delta}(z_2)^{\top}\right]^{-1}\nonumber\\
        &-\left[I_{d_2}-(I_{d_2}-\eta\nabla_{yy}^2g(x_1,y))\nabla\mathcal{P}_{\mathcal{Y}\delta}(z_1)^{\top}\right]^{-1}\|\|\nabla_yf(x_2,y)\|\nonumber\\
        &+\eta \|\nabla_{xy}^2g(x_1,y)\| \|\nabla\mathcal{P}_{\mathcal{Y}\delta}(z_1)\|\|\left[I_{d_2}-(I_{d_2}-\eta\nabla_{yy}^2g(x_1,y))\nabla\mathcal{P}_{\mathcal{Y}\delta}(z_1)^{\top}\right]^{-1}\|\|\nabla_yf(x_2,y)-\nabla_yf(x_1,y)\|
    \end{align}
    We have 
    \begin{align}
        &\|\nabla_xf(x_1,y)-\nabla_x f(x_2,y)\|\leq L_{f}\|x_2-x_1\|\\
        &\|\nabla_yf(x_1,y)-\nabla_y f(x_2,y)\|\leq L_{f}\|x_2-x_1\|\\
        &\|\nabla_{xy}^2g(x_2,y)-\nabla_{xy}^2g(x_1,y)\|\leq L_{gxy}\|x_2-x_1\|
    \end{align}
    \begin{align}
        &\|\nabla\mathcal{P}_{\mathcal{Y}\delta}(z_2)-\nabla\mathcal{P}_{\mathcal{Y}\delta}(z_1)\|\nonumber\\
        \leq&\frac{cd_2}{\delta}\|z_2-z_1\|\nonumber\\
        =&\frac{cd_2}{\delta}\|y-\eta\nabla_yg(x_2,y)-y+\eta\nabla_yg(x_1,y)\|\nonumber\\
        \leq&\eta \frac{cd_2}{\delta}\|\nabla_yg(x_1,y)-\eta\nabla_yg(x_2,y)\|\nonumber\\
        \leq&\eta \frac{cd_2}{\delta}L_{g}\|x_1-x_2\|
    \end{align}
    \begin{align}
        &\|\left[I_{d_2}-(I_{d_2}-\eta\nabla_{yy}^2g(x_2,y))\nabla\mathcal{P}_{\mathcal{Y}\delta}(z_2)^{\top}\right]^{-1}-\left[I_{d_2}-(I_{d_2}-\eta\nabla_{yy}^2g(x_1,y))\nabla\mathcal{P}_{\mathcal{Y}\delta}(z_1)^{\top}\right]^{-1}\|\nonumber\\
        \leq&\frac{1}{\eta^2\mu_g^2}\|(I_{d_2}-\eta\nabla_{yy}^2g(x_1,y))\nabla\mathcal{P}_{\mathcal{Y}\delta}(z_1)^{\top}-(I_{d_2}-\eta\nabla_{yy}^2g(x_2,y))\nabla\mathcal{P}_{\mathcal{Y}\delta}(z_2)^{\top}\|\nonumber\\
        \leq&\frac{1}{\eta^2\mu_g^2}\|\nabla\mathcal{P}_{\mathcal{Y}\delta}(z_1)-\nabla\mathcal{P}_{\mathcal{Y}\delta}(z_2)\|\|I_{d_2}-\eta\nabla_{yy}^2g(x_1,y)\|\nonumber\\
        &+\frac{1}{\eta^2\mu_g^2}\|\nabla\mathcal{P}_{\mathcal{Y}\delta}(z_2)\|\|I_{d_2}-\eta\nabla_{yy}^2g(x_1,y)-I_{d_2}+\eta\nabla_{yy}^2g(x_2,y)\|\nonumber\\
        \leq&\left(\frac{1}{\eta\mu_g^2}\frac{cd_2}{\delta}L_{g}(1-\eta\mu_g)+\frac{L_{gyy}}{\eta\mu_g^2}\right)\|x_1-x_2\|
    \end{align}
    
    Therefore, we have 
    \begin{align}
        &\|\nabla f_{\delta}(x_1,y)-\nabla f_{\delta}(x_2,y)\|\nonumber\\
        \leq&\left(L_f+ \frac{L_{gxy}C_{fy}}{\mu_g}+ \frac{C_{gxy}C_{fy}}{\mu_g}\eta\frac{cd_2}{\delta}L_{g}+  C_{gxy}C_{fy}\left(\eta\frac{1}{\mu_g^2}\frac{cd_2}{\delta}L_{g}(1-\eta\mu_g)+\frac{L_{gyy}}{\mu_g^2}\right)+ \frac{C_{gxy}}{\mu_g}L_{f}\right)\|x_1-x_2\|
    \end{align}
\end{proof}


\subsection{Proof of Lemma \ref{lemma:lip_f_x}}
Here we first give a detailed version of Lemma \ref{lemma:lip_f_x} and then give the proof.
\begin{lemma} (\textbf{Lipschitz continous of $\nabla F_{\delta}(x)$.})
    Under Assumptions \ref{assump:upper_level}, \ref{assump:lower_level} and Lemma \ref{lemma:bound_of_nablaP}, we have $\nabla F_{\delta}(x)$ is Lipschitz continuous w.r.t $x$, such that
    \begin{align}
         &\|\nabla F_{\delta}(x_1)-\nabla F_{\delta}(x_2)\|
         \leq L_{F_{\delta}}(\frac{d_2}{\delta})\|x_1-x_2\|
    \end{align}
    where $L_{F_{\delta}}(\frac{d_2}{\delta})=\frac{L_g}{\mu_g}L_1(\frac{d_2}{\delta})+L_2(\frac{d_2}{\delta})$, $L_1(\frac{d_2}{\delta})=L_f+ \frac{L_{gxy}C_{fy}}{\mu_g}+ \frac{C_{gxy}C_{fy}}{\mu_g}(\frac{cd_2}{\delta}+\eta \frac{cd_2}{\delta}L_{g})+  C_{gxy}C_{fy}(\frac{1}{\eta\mu_g^2}(\frac{cd_2}{\delta}+\eta \frac{cd_2}{\delta}L_{g})(1-\eta\mu_g)+\frac{L_{gyy}}{\mu_g^2})+ \frac{C_{gxy}}{\mu_g}L_{f}$ and $L_2(\frac{d_2}{\delta})=L_f+ \frac{L_{gxy}C_{fy}}{\mu_g}+ \frac{C_{gxy}C_{fy}}{\mu_g}\eta\frac{cd_2}{\delta}L_{g}+  C_{gxy}C_{fy}(\eta\frac{1}{\mu_g^2}\frac{cd_2}{\delta}L_{g}(1-\eta\mu_g)+\frac{L_{gyy}}{\mu_g^2})+ \frac{C_{gxy}}{\mu_g}L_{f}$.
\end{lemma}

\begin{proof}
    We have 
    \begin{align}
        &\|\nabla F_{\delta}(x_1)-\nabla F_{\delta}(x_2)\|\nonumber\\
        =&\|\nabla f_{\delta}(x_1,y^*(x_1))-\nabla f_{\delta}(x_2,y^*(x_2))\|\nonumber\\
        \leq&\|\nabla f_{\delta}(x_1,y^*(x_1))-\nabla f_{\delta}(x_1,y^*(x_2))\|+\|f_{\delta}(x_1,y^*(x_2))-f_{\delta}(x_2,y^*(x_2))\|
    \end{align}
    For the first term, we have
    \begin{align}
        \|\nabla f_{\delta}(x_1,y^*(x_1))-\nabla f_{\delta}(x_1,y^*(x_2))\|\leq L_1(\frac{d_2}{\delta})\|y^*(x_1)-y^*(x_2)\|\leq \frac{L_{g}}{\mu_g}L_1(\frac{d_2}{\delta})\|x_1-x_2\|
    \end{align}

    Thus, we have 
    \begin{align}
         &\|\nabla F_{\delta}(x_1)-\nabla F_{\delta}(x_2)\|
         \leq \left(L_2(\frac{d_2}{\delta})+\frac{L_g}{\mu_g}L_1(\frac{d_2}{\delta})\right)\|x_1-x_2\|
    \end{align}
\end{proof}

\section{Proof of Lemma \ref{lemma:bound_of_nablaP}}

\begin{proof}
    By the definition of $\bar{H}$, we have 
    \begin{align}
        \bar{H}=&\sum_{i=1}^{d_2}\frac{1}{2\delta}\left( \mathcal{P}_{\mathcal{Y}}(z+\delta u_i)- \mathcal{P}_{\mathcal{Y}}(z-\delta u_i) \right)u_i^{\top}
        =\begin{bmatrix}
            \sum_{i=1}^{d_2}\frac{1}{2\delta}\left( \mathcal{P}_{\mathcal{Y}}^1(z+\delta u_i)- \mathcal{P}_{\mathcal{Y}}^1(z-\delta u_i) \right)u_i^{\top}\\
            \vdots\\
            \sum_{i=1}^{d_2}\frac{1}{2\delta}\left( \mathcal{P}_{\mathcal{Y}}^{d_2}(z+\delta u_i)- \mathcal{P}_{\mathcal{Y}}^{d_2}(z-\delta u_i) \right)u_i^{\top}
        \end{bmatrix}
    \end{align}
    According to \cite{lin2022gradient}, for each element $\bar{H}_j=\sum_{i=1}^{d_2}\frac{1}{2\delta}\left( \mathcal{P}_{\mathcal{Y}}^j(z+\delta u_i)- \mathcal{P}_{\mathcal{Y}}^j(z-\delta u_i) \right)u_i^{\top}$, we have 
    \begin{align}
        \mathbb{E}[\frac{1}{d_2}\sum_{i=1}^{d_2}\frac{d_2}{2\delta}\left( \mathcal{P}_{\mathcal{Y}}^j(z+\delta u_i)- \mathcal{P}_{\mathcal{Y}}^j(z-\delta u_i) \right)u_i^{\top}]=\nabla \mathcal{P}_{\mathcal{Y}\delta}^j(z).
    \end{align}
    Therefore, we have $\mathbb{E}[\bar{H}]=\nabla \mathcal{P}_{\mathcal{Y}\delta}(z)$. According to \cite{lin2022gradient}, we have $\mathbb{E}_{u_i}[\|\frac{d_2}{2\delta}\left( \mathcal{P}_{\mathcal{Y}}^j(z+\delta u_i)- \mathcal{P}_{\mathcal{Y}}^j(z-\delta u_i) \right)u_i^{\top}\|^2]\leq 16\sqrt{2\pi}d_2L_p^2$. Then, we have 
    \begin{align}
       &\mathbb{E}[\|\bar{H}\|^2]\nonumber\\
       \leq& \mathbb{E}[\|\bar{H}\|_F^2]=\mathbb{E}[\sum_{j=1}^{d_2}\|\bar{H}_j\|^2]=\sum_{j=1}^{d_2}\mathbb{E}[\|\bar{H}_j\|^2]\nonumber
       \\
       =&\sum_{j=1}^{d_2}\mathbb{E}\left[\left\|\frac{1}{d_2}\sum_{i=1}^{d_2}\frac{d_2}{2\delta}\left( \mathcal{P}_{\mathcal{Y}}^j(z+\delta u_i)- \mathcal{P}_{\mathcal{Y}}^j(z-\delta u_i) \right)u_i^{\top}\right\|^2\right]\nonumber\\
       \leq&\sum_{j=1}^{d_2}\frac{1}{d_2^2}\sum_{i=1}^{d_2}\mathbb{E}_{u_i}\left[\left\|\frac{d_2}{2\delta}\left( \mathcal{P}_{\mathcal{Y}}^j(z+\delta u_i)- \mathcal{P}_{\mathcal{Y}}^j(z-\delta u_i) \right)u_i^{\top}\right\|^2\right]\nonumber\\
       \leq&d_2\cdot\frac{1}{d_2^2} d_2\cdot16\sqrt{2\pi}d_2L_p^2\nonumber\\
       =&16\sqrt{2\pi}d_2L_p^2
    \end{align}
    That completes the proof.
\end{proof}

\section{Proof of Lemma \ref{lemma:bounds_of_hypergrad}}
\begin{proof}
    For convenience, define $\bar{G}_{yy}=Q\prod_{i=1}^{c(Q)}\left( (I_{d_2}-\eta\nabla_{yy}^2g(x,y))\bar{H}(z;u^i)^{\top}\right)$ and $G_{yy}=\left[I_{d_2}-(I_{d_2}-\eta\nabla_{yy}^2g(x,y))\nabla\mathcal{P}_{\mathcal{Y}\delta}(z)^{\top}\right]^{-1}$. We set $\eta<\frac{1}{\mu_g}$
    We have
	\begin{align}
		&\mathbb{E}_{\bar{\xi}}[\bar{\nabla} f_{\delta}(x,y;\bar{\xi})]
		=\nabla_x f(x,y)-\eta\nabla_{xy}^2g(x,y)\nabla\mathcal{P}_{\mathcal{Y}\delta}(z)^{\top}\mathbb{E}\left[\bar{G}_{yy}\right]\nabla_yf(x,y)
	\end{align}
	We have 
	\begin{align}
		&\left\|\nabla f_{\delta}(x,y)-\mathbb{E}[\bar{\nabla} f_{\delta}(x,y;\bar{\xi})]\right\|\nonumber\\
		=&\left\|\eta
		\nabla_{xy}^2g(x,y)\partial_z\mathcal{P}_{\mathcal{Y}\delta}(z)^{\top}\left\{\mathbb{E}\left[\bar{G}_{yy}\right]-G_{yy}\right\}\nabla_yf(x,y) \right\|\nonumber\\
        \leq&\eta C_{gxy}C_{fy}
        \left\|\mathbb{E}\left[\bar{G}_{yy}\right]-G_{yy}\right\|
	\end{align}
	where the third inequality is due to the non-expansive of the projector operation.
	
	Due to the independency of $u,c(Q)$, we have 
	\begin{align}
        &\mathbb{E}\left[\bar{G}_{yy}\right]\nonumber\\
        =&\mathbb{E}\left[Q\prod_{i=1}^{c(Q)}\left( (I_{d_2}-\eta\nabla_{yy}^2g(x,y))\bar{H}(z;u^i)^{\top}\right)\right]\nonumber\\
		=&Q\mathbb{E}_{c(Q)}\left[\mathbb{E}_{u}\left[\prod_{i=1}^{c(Q)}\left( (I_{d_2}-\eta\nabla_{yy}^2g(x,y))\bar{H}(z;u^i)^{\top}\right)\right]\right]\nonumber\\
		=&Q\mathbb{E}_{c(Q)}\left[(I_{d_2}-\eta\nabla_{yy}^2g(x,y))\nabla\mathcal{P}_{\mathcal{Y}\delta}(z)^{\top}\right]^{c(Q)}\nonumber\\
        =&\sum_{i=0}^{Q-1}\left[(I_{d_2}-\eta\nabla_{yy}^2g(x,y))\nabla\mathcal{P}_{\mathcal{Y}\delta}(z)^{\top}\right]^i
	\end{align}

	In addition, we have 
	\begin{align}
        &G_{yy}\nonumber\\
		=&\left[I_{d_2}-(I_{d_2}-\eta\nabla_{yy}^2g(x,y))\nabla\mathcal{P}_{\mathcal{Y}\delta}(z)^{\top}\right]^{-1}\nonumber\\
		=&\sum_{i=0}^{\infty}\left[(I_{d_2}-\eta\nabla_{yy}^2g(x,y))\nabla\mathcal{P}_{\mathcal{Y}\delta}(z)^{\top}\right]^i\nonumber\\
		=&\mathbb{E}\left[Q\prod_{i=1}^{c(Q)}\left( (I_{d_2}-\eta\nabla_{yy}^2g(x,y))\bar{H}(z;u^i)^{\top}\right)\right]
        +\sum_{i=Q}^{\infty}\left[(I_{d_2}-\eta\nabla_{yy}^2g(x,y))\nabla\mathcal{P}_{\mathcal{Y}\delta}(z)^{\top}\right]^i\nonumber\\
	\end{align}
	where $z=y-\eta\nabla_{y}g(x,y)$,
	which implies that 
	
		\begin{align}
            &\left\|\mathbb{E}\left[\bar{G}_{yy}\right]-\bar{G}_{yy}\right\|\nonumber\\
			=&\left\|\left[(I_{d_2}-(I_{d_2}-\eta\nabla_{yy}^2g(x,y))\nabla\mathcal{P}_{\mathcal{Y}\delta}(z)^{\top}\right]^{-1}-\mathbb{E}\left[Q\prod_{i=1}^{c(Q)}\left( (I_{d_2}-\eta\nabla_{yy}^2g(x,y;\zeta^i))\bar{H}(z;u^i)^{\top}\right)\right]\right\|\nonumber\\
			\leq&\sum_{i=Q}^{\infty}\left\|(I_{d_2}-\eta\nabla_{yy}^2g(x,y))\nabla\mathcal{P}_{\mathcal{Y}\delta}(z)^{\top}\right\|^i\nonumber\\\
			\leq&\sum_{i=Q}^{\infty}\left\|I_{d_2}-\eta\nabla_{yy}^2g(x,y))\right\|^i\left\|\nabla\mathcal{P}_{\mathcal{Y}\delta}(z)\right\|^i\nonumber\\
			\leq&\frac{1}{\eta\mu_g}(1-\eta\mu_g)^Q
		\end{align}

	Thus, we have 
	\begin{align}
		&\left\|\nabla f_{\delta}(x,y)-\mathbb{E}[\bar{\nabla} f_{\delta}(x,y;\bar{\xi})]\right\|
		\leq\frac{ C_{gxy}C_{fy}}{\mu_g}(1-\eta\mu_g)^Q
	\end{align}
	
	Then, we prove the bound on the variance.
	 we have 
	\begin{align}
		&\mathbb{E}\left[\left \|\bar{\nabla} f_{\delta}(x,y;\bar{\xi})-\mathbb{E}\left[\bar{\nabla} f_{\delta}(x,y;\bar{\xi})\right]\right\|^2\right]\nonumber\\
		=&\mathbb{E}\left[\left\|\eta \nabla_{xy}^2g(x,y) \bar{H}(z;u^0)^{\top}\bar{G}_{yy}\nabla_yf(x,y)- \eta  \nabla_{xy}^2g(x,y) \nabla\mathcal{P}_{\mathcal{Y}\delta}(z)^{\top}\mathbb{E}[\bar{G}_{yy}]\nabla_yf(x,y)  \right\|^2\right]\nonumber\\
		\leq&2\eta^2\left\|\nabla_{xy}^2g(x,y)\right\|^2\mathbb{E}\left[\left\|  \bar{H}(z;u^0)- \nabla\mathcal{P}_{\mathcal{Y}\delta}(z)\right\|^2\right]\mathbb{E}\left[\left\|\bar{G}_{yy}\right\|^2\right]\left\|\nabla_yf(x,y)\right\|^2\nonumber\\
		&+2\eta^2\left\|\nabla_{xy}g(x,y)\right\|^2\left\| \nabla\mathcal{P}_{\mathcal{Y}\delta}(z)\right\|^2\mathbb{E}\left[\left\|\bar{G}_{yy}-\mathbb{E}[\bar{G}_{yy}]\right\|^2\right]\left\|\nabla_yf(x,y)\right\|^2
	\end{align}
	
	For the first term in the above inequality, we have
	\begin{align}
	    \mathbb{E}\left[\left\|  \bar{H}(z;u^0)- \nabla\mathcal{P}_{\mathcal{Y}\delta}(z)\right\|^2\right]\leq32\sqrt{2\pi}d_2L_p^2+2
	\end{align}
    
	For $\mathbb{E}[\|\bar{G}_{yy}\|^2]$, we have 
	\begin{align}
		\mathbb{E}[\|\bar{G}_{yy}\|^2]
        =&\sum_{q=0}^{Q-1}\mathbb{E}\left[\left\|\prod_{i=1}^{q}\left( (I_{d_2}-\eta\nabla_{yy}^2g(x,y))\bar{H}(z;u^i)^{\top}\right)\right\|^2\right]\nonumber\\
        \leq&\sum_{q=0}^{Q-1}((1-\eta\mu_g)4(2\pi)^{1/4}\sqrt{d_2}L_p)^{2q}\nonumber\\
        \leq&\frac{1}{1-(1-\eta\mu_g)^216\sqrt{2\pi}d_2L_p^2}
	\end{align}
	where the last inequality is obtained by setting $\frac{1}{\mu_g}(1-\frac{1}{4(2\pi)^{1/4}\sqrt{d_2}L_p})\leq\eta<\frac{1}{\mu_g}$.

    We can also derive that 
	 \begin{align}
        &\left\|\mathbb{E}\left[\bar{G}_{yy}\right]\right\|\nonumber\\
        =&\|\sum_{i=0}^{Q-1}\left[(I_{d_2}-\eta\nabla_{yy}^2g(x,y))\nabla\mathcal{P}_{\mathcal{Y}\delta}(z)^{\top}\right]^i\|\nonumber\\
        \leq&\sum_{i=0}^{Q-1}\|(I_{d_2}-\eta\nabla_{yy}^2g(x,y))\nabla\mathcal{P}_{\mathcal{Y}\delta}(z)^{\top}      \|^i\nonumber\\
        \leq&\sum_{i=0}^{Q-1}(1-\eta\mu_g)^i\nonumber\\
        \leq&\frac{1}{\eta\mu_g}
    \end{align}
     Then, we have 
    \begin{align}
        \mathbb{E}\left[\left\|G_{yy}-\mathbb{E}[G_{yy}]\right\|^2\right]\leq\frac{2}{1-(1-\eta\mu_g)^216\sqrt{2\pi}d_2L_p^2}+\frac{2}{\eta^2\mu_g^2} 
    \end{align}
	Therefore, combining the above inequalities,  we can bound the variance as follows,
\begin{align}
    &\mathbb{E}\left[\left \|\bar{\nabla}f(x,y;\bar{\xi})-\mathbb{E}\left[\bar{\nabla}f(x,y;\bar{\xi})\right]\right\|^2\right]
    \leq\sigma_f(d_2)
\end{align}
where $\sigma_f(d_2)=2\eta^2C_{gxy}^2(32\sqrt{2\pi}d_2L_p^2+2)C_{fy}^2\frac{1}{1-(1-\eta\mu_g)^216\sqrt{2\pi}d_2L_p^2}+2\eta^2C_{gxy}^2(\frac{2}{1-(1-\eta\mu_g)^216\sqrt{2\pi}d_2L_p^2}+\frac{2}{\eta^2\mu_g^2})C_{fy}^2$
	That completes the proof.
\end{proof}

\section{Route Map of Our Convergence Analysis}{\label{section:route_map}}
Here we give a simple route map of our convergence analysis.
\begin{figure}[h]
    \includegraphics[width=1\textwidth]{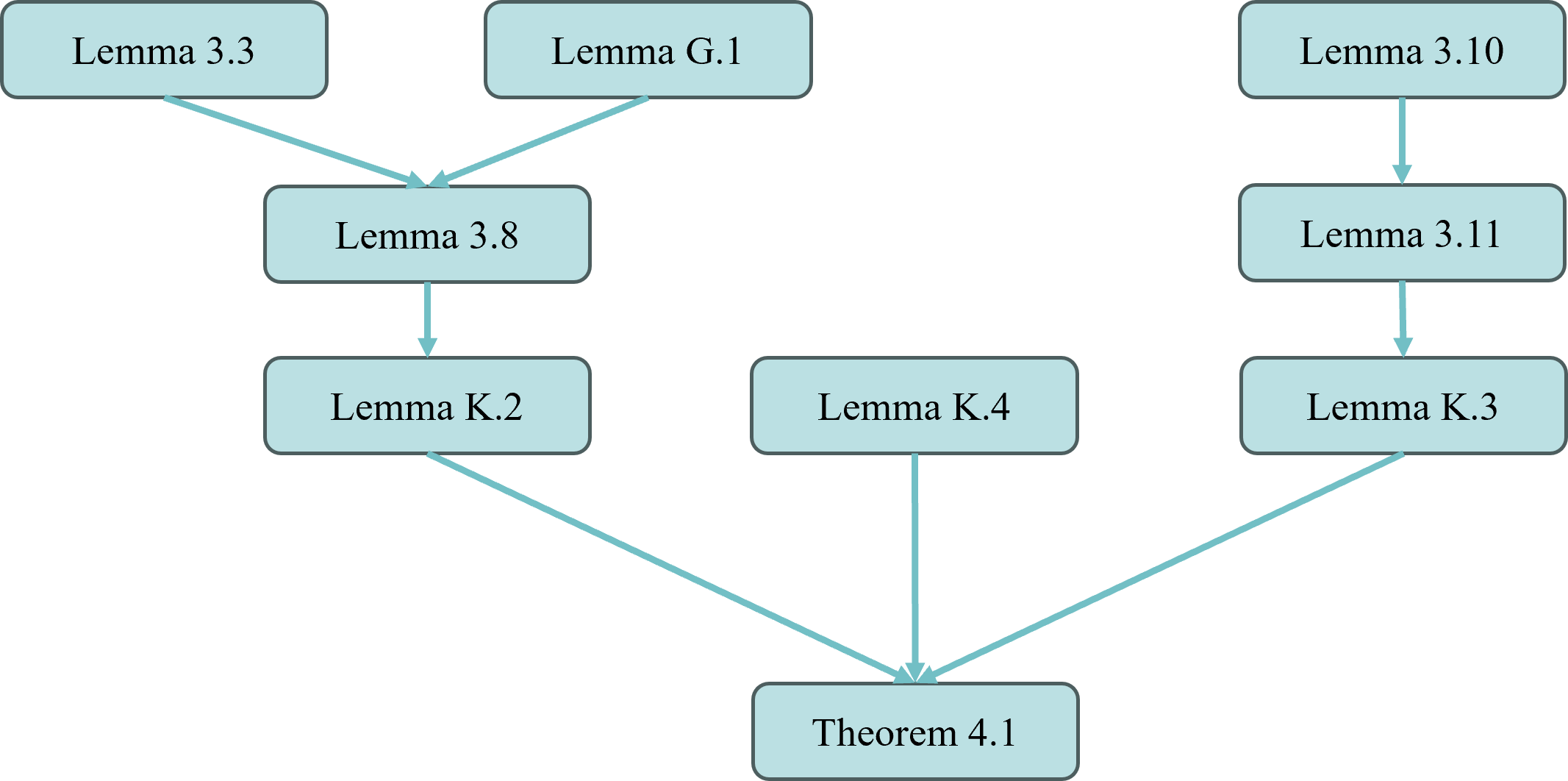}
     \caption{Route map of convergence analysis.}
    \label{fig:my_label}
\end{figure}

\section{Deriving the Convergence Metric}
In this section, we give a detailed analysis to derive the metric. Based  on the Lipshitz continuous of $\nabla F_{\delta}(x)$, we have the following lemma,
\begin{lemma}\label{lemma:boud_of_F_w}
Under Assumptions \ref{assump:upper_level}, \ref{assump:lower_level} and Lemma \ref{lemma:bound_of_nablaP}, we have 
    \begin{align}
        \|\nabla F_{\delta}(x_k)-w_{k}\|^2\leq2L_1^2(\frac{d_2}{\delta})\|y^*(x_k)-y_k\|^2+2\|\nabla f_{\delta}(x_k,y_k)-w_{k}\|^2.
    \end{align}
\end{lemma}
\begin{proof}
    We have 
    \begin{align}
        &\|\nabla F_{\delta}(x_k)-w_{k}\|^2\nonumber\\
        =&\|\nabla f_{\delta}(x_k,y^*(x_k))-\nabla f_{\delta}(x_k,y_k)+\nabla f_{\delta}(x_k,y_k)-w_{k}\|^2\nonumber\\
        \leq&2\|\nabla f_{\delta}(x_k,y^*(x_k))-\nabla f_{\delta}(x_k,y_k)\|^2+2\|\nabla f_{\delta}(x_k,y_k)-w_{k}\|^2\nonumber\\
        \leq&2L_1^2(\frac{d_2}{\delta})\|y^*(x_k)-y_k\|^2+2\|\nabla f_{\delta}(x_k,y_k)-w_{k}\|^2
    \end{align}
\end{proof}

\subsection{Usefull Lemmas in Convergence Rate}
\begin{lemma}\label{lemma:reduction_F_delta} (\textbf{Descent on the function value.})
	Under Assumptions \ref{assump:upper_level},  \ref{assump:lower_level}, and Lemma \ref{lemma:lip_f_x}, let $F_{\delta}(x)$ be an approximation function of $F(x)$ and have the gradient $\nabla F_{\delta}(x_k)$, and $\gamma\eta_k\leq\frac{1}{2L_{F_{\delta}}(\frac{d_2}{\delta}) c_l}$, we have 
    \begin{align}
        &F_{\delta}(x_{k+1})
        \leq F_{\delta}(x_{k})+\eta_k\gamma c_l\|\nabla F_{\delta}(x_{k})-w_{k}\|^2-\frac{\eta_k}{2\gamma c_l}\|\tilde{x}_{k+1}-x_k\|^2
    \end{align}
\end{lemma}
\begin{proof}
    Due to the smoothness of $F_{\delta}$ and let $\tilde{x}_{k+1}=x_k-\frac{\gamma}{\mathcal{P}_{[1/c_u,1/c_l]}(\sqrt{m_{2,k}}+G_0)}w_{k}$, we have 
    \begin{align}
		&F_{\delta}(x_{k+1})\nonumber\\
        \leq&F_{\delta}(x_k)+ \nabla F_{\delta}(x_k)^{\top}(x_{k+1}-x_k)+\frac{1}{2}L_{F_{\delta}}(\frac{d_2}{\delta})\|x_{k+1}-x_k\|^2\nonumber\\
        =&F_{\delta}(x_k)+ \eta_k\nabla F_{\delta}(x_k)^{\top}(\tilde{x}_{k+1}-x_k)+\frac{1}{2}L_{F_{\delta}}(\frac{d_2}{\delta})\|\eta_k(\tilde{x}_{k+1}-x_k)\|^2\nonumber\\
        =&F_{\delta}(x_k)+ \eta_k \langle w_{k},\tilde{x}_{k+1}-x_k\rangle +\eta_k\langle\nabla F_{\delta}(x_k)-w_{k},\tilde{x}_{k+1}-x_k\rangle+\frac{\eta_k^2}{2}L_{F_{\delta}}(\frac{d_2}{\delta})\|\tilde{x}_{k+1}-x_k\|^2\
    \end{align}
    In our algorithm, we have $\tilde{x}_{k+1}=x_k-\frac{\gamma}{\mathcal{P}_{[1/c_u,1/c_l]}(\sqrt{m_{2,k}}+G_0)}w_{k}=\mathop{\arg\min}_{x\in\mathbb{R}^{d_1}}\frac{1}{2}\|x-x_k+\frac{\gamma}{\mathcal{P}_{[1/c_u,1/c_l]}(\sqrt{m_{2,k}}+G_0)}w_{k}\|^2$. We have the following optimal condition,
    \begin{align}
        \langle\tilde{x}_{k+1}-x_k+\frac{\gamma}{\mathcal{P}_{[1/c_u,1/c_l]}(\sqrt{m_{2,k}}+G_0)}w_{k}, x-\tilde{x}_{k+1}\rangle\geq 0, \ x\in\mathbb{R}^{d_1}
    \end{align}
    Set $x=x_k$, we can obtain 
    \begin{align}
       \gamma c_l\langle w_{k}, \tilde{x}_{k+1}-x_k\rangle\leq-\|\tilde{x}_{k+1}-x_k\|^2
    \end{align}
    Thus, we have
    \begin{align}
        \langle w_{k}, \tilde{x}_{k+1}-x_k\rangle\leq-\frac{1}{\gamma c_l}\|\tilde{x}_{k+1}-x_k\|^2
    \end{align}
    In addition, we can obtain 
    \begin{align}
        &\langle\nabla F_{\delta}(x_{k})-w_{k}, \tilde{x}_{k+1}-x_k\rangle\nonumber\\
        \leq&\|\nabla F_{\delta}(x_{k})-w_{k}\|_2\|\tilde{x}_{k+1}-x_k\|_2\nonumber\\
        \leq&\gamma c_l\|\nabla F_{\delta}(x_{k})-w_{k}\|^2+\frac{1}{4\gamma c_l}\|\tilde{x}_{k+1}-x_k\|^2
    \end{align}
    Then, setting $\gamma\leq\frac{1}{2L_{F_{\delta}}(\frac{d_2}{\delta}) c_l\eta_k}$, we can derive  
    \begin{align}
        &F_{\delta}(x_{k+1})\nonumber\\
        \leq&F_{\delta}(x_{k})+\eta_k\gamma c_l\|\nabla F_{\delta}(x_{k})-w_{k}\|^2+\frac{\eta_k}{4\gamma c_l}\|\tilde{x}_{k+1}-x_k\|^2-\frac{\eta_k}{\gamma c_l}\|\tilde{x}_{k+1}-x_k\|^2+\frac{\eta_k^2}{2}L_{F_{\delta}}(\frac{d_2}{\delta})\|\tilde{x}_{k+1}-x_k\|^2\nonumber\\
        \leq&F_{\delta}(x_{k})+\eta_k\gamma c_l\|\nabla F_{\delta}(x_{k})-w_{k}\|^2-\frac{\eta_k}{2\gamma c_l}\|\tilde{x}_{k+1}-x_k\|^2
    \end{align}
\end{proof}

\begin{lemma}\label{lemma:convergence_of_y} (\textbf{Error between the updates of $y$ and the optimal solution $y^*$})
	Under Assumptions \ref{assump:upper_level}, \ref{assump:lower_level}, let $\tilde{y}_{k+1}=\mathcal{P}_{\mathcal{Y}}(y_{k}-\frac{\tau}{\mathcal{P}_{[1/c_u,1/c_l]}(\sqrt{m_{1,k}}+G_0)}  v_{k})$ and $\eta_k\leq1$, $\tau\leq\frac{1}{6L_gc_u}$, we have 
	\begin{align}
		&\|y_{k+1}-y^*(x_{k+1})\|^2\nonumber\\
        \leq&(1-\frac{\eta_k\tau\mu_g c_u}{4})\|y^*(x_k)-y_k\|^2+\frac{25\eta_k\tau c_u}{6\mu_g}\|\nabla_yg(x_k,y_k)-v_{k}\|^2-\frac{3\eta_k}{4}\|\tilde{y}_{k+1}-y_k\|^2+\frac{25L_y^2\eta_k}{6\tau\mu_g c_u}\|x_{k}-\tilde{x}_{k+1}\|^2
	\end{align}
\end{lemma}

\begin{proof}
		Define $\tilde{y}_{k+1}=\mathcal{P}_{\mathcal{Y}}(y_{k}-\frac{\tau}{\mathcal{P}_{[1/c_u,1/c_l]}(\sqrt{m_{1,k}}+G_0)}  v_{k})$. We have $y_{k+1}=(1-\eta_k)y_{k}+\eta_k\tilde{y}_{k+1}$.
	
	According to the strong convexity of $g$, we have 
	\begin{align}
		&g(x_k,y)\nonumber\\
  \geq& g(x_k,y_k)+\langle \nabla_y g(x_k,y_k),y-y_k \rangle+\frac{\mu_g}{2}\|y-y_k\|^2\nonumber\\
		=& g(x_k,y_k)+\langle v_{k}, y-\tilde{y}_{k+1}\rangle+\langle \nabla_yg(x_k,y_k)-v_{k}, y-\tilde{y}_{k+1} \rangle + \langle \nabla_yg(x_k,y_k),\tilde{y}_{k+1}-y_k \rangle + \frac{\mu_g}{2}\|y-y_k\|^2
	\end{align}
	According to the smoothness of $g$, we have
	\begin{align}
		g(x_k, \tilde{y}_{k+1})\leq g(x_k,y_{k})+\langle \nabla_yg(x_k,y_k),\tilde{y}_{k+1}-y_k \rangle +\frac{L_g}{2}\|\tilde{y}_{k+1}-y_k\|^2
	\end{align}
	Then, combining the above inequalities, we have 
	\begin{align}
		g(x_k, y)\geq g(x_k,\tilde{y}_{k+1})-\frac{L_g}{2}\|\tilde{y}_{k+1}-y_k\|^2+\langle v_{k}, y-\tilde{y}_{k+1} \rangle +\langle \nabla_y g(x_k,y_k)-v_{k},y-\tilde{y}_{k+1} \rangle+\frac{\mu_g}{2}\|y-y_k\|^2
	\end{align}
	
	In our Algorithm \ref{alg:DMSCBO}, we have 
	\begin{align}
		\tilde{y}_{k+1}=&\mathcal{P}_{\mathcal{Y}}\left(y_k-\frac{\tau}{\mathcal{P}_{[1/c_u,1/c_l]}(\sqrt{m_{1,k}}+G_0)} v_{k}\right)
		=\mathop{\arg\min}_{y\in\mathcal{Y}}\frac{1}{2}\left\|y-y_k+\frac{\tau}{\mathcal{P}_{[1/c_u,1/c_l]}(\sqrt{m_{1,k}}+G_0)} v_{k}\right\|^2.
	\end{align}
	Since $\mathcal{Y}$ is a convex set and the function $\frac{1}{2}\left\|y-y_k+\frac{\tau}{\mathcal{P}_{[1/c_u,1/c_l]}(\sqrt{\|v_{k}\|})} v_{k}\right\|^2$ is convex, we have 
	\begin{align}
		\left\langle \tilde{y}_{k+1}-y_k+\frac{\tau}{\mathcal{P}_{[1/c_u,1/c_l]}(\sqrt{m_{1,k}}+G_0)} v_{k}, y-\tilde{y}_{k+1} \right\rangle\geq0, \quad y\in\mathcal{Y}
	\end{align}
	Then we have 
	\begin{align}
	      \tau c_u\langle v_{k}, y-\tilde{y}_{k+1} \rangle \geq\langle \tilde{y}_{k+1}-y_k,\tilde{y}_{k+1}-y \rangle
	\end{align}
	
	Then we have 
	\begin{align}
		&g(x_k,y)\nonumber\\
		\geq&g(x_k,\tilde{y}_{k+1})-\frac{L_g}{2}\|\tilde{y}_{k+1}-              y_k\|^2+\frac{\mu_g}{2}\|y-y_k\|^2+\langle  \nabla_yg(x_k,y_k)-        v_{k}, y-\tilde{y}_{k+1}\rangle+\frac{1}{\tau c_u}\langle                \tilde{y}_{k+1}-y_k,\tilde{y}_{k+1}-y \rangle
	\end{align}
	Let $y=y^*(x_k)$. Since $g(x_k,y^*(x_k))\leq g(x_k,\tilde{y}_{k+1})$, we have
	\begin{align}
		&g(x_k,\tilde{y}_{k+1})\geq g(x_k,y^*(x_k))\nonumber\\
		\geq&g(x_k,\tilde{y}_{k+1})-\frac{L_g}{2}\|\tilde{y}_{k+1}-y_k\|^2+\frac{\mu_g}{2}\|y^*(x_k)-y_k\|^2+\langle  \nabla_yg(x_k,y_k)-v_{k}, y^*(x_k)-\tilde{y}_{k+1}\rangle\nonumber\\
		&+\frac{1}{\tau c_u}\| \tilde{y}_{k+1}-y_k \|^2+\frac{1}{\tau c_u}\langle \tilde{y}_{k+1}-y_k,y_k-y^*(x_k) \rangle
	\end{align}
	
	In addition, we have
	\begin{align}
		&\langle  \nabla_yg(x_k,y_k)-v_{k}, y^*(x_k)-\tilde{y}_{k+1}\rangle\nonumber\\
		=&\langle  \nabla_yg(x_k,y_k)-v_{k}, y^*(x_k)-y_{k}\rangle + \langle  \nabla_yg(x_k,y_k)-v_{k}, y_k-\tilde{y}_{k+1}\rangle \nonumber\\
		\geq &-\|\nabla_yg(x_k,y_k)-v_{k}\|\|y^*(x_k)-y_{k}\| -  \|\nabla_yg(x_k,y_k)-v_{k}\|\|y_k-\tilde{y}_{k+1}\|\nonumber\\
		\geq&-\frac{1}{\mu_g}\|\nabla_yg(x_k,y_k)-v_{k}\|^2-\frac{\mu_g}{4}\|y^*(x_k)-y_{k}\|^2-\frac{1}{\mu_g}\|\nabla_yg(x_k,y_k)-v_{k}\|^2-\frac{\mu_g}{4}\|y_k-\tilde{y}_{k+1}\|^2\nonumber\\
		\geq&-\frac{2}{\mu_g}\|\nabla_yg(x_k,y_k)-v_{k}\|^2-\frac{\mu_g}{4}\|y^*(x_k)-y_{k}\|^2-\frac{\mu_g}{4}\|y_k-\tilde{y}_{k+1}\|^2
	\end{align}
	The first inequality is due to $\langle a,b\rangle\geq -\|a\|\|b\|$ and the second inequality is due to the Young's inequality. We also have
	\begin{align}
		&\|y_{k+1}-y^*(x_k)\|^2\nonumber\\
		\leq&\|y_k+\eta_k(\tilde{y}_{k+1}-y_k)-y^*(x_k)\|^2\nonumber\\
		=&\|y_k-y^*(x_k)\|^2+\eta_k^2\|\tilde{y}_{k+1}-y_k\|^2+2\eta_k\langle\tilde{y}_{k+1}-y_k, y_k-y^*(x_k)\rangle
	\end{align}
	Therefore, we have 
	\begin{align}
		\langle\tilde{y}_{k+1}-y_k, y_k-y^*(x_k)\rangle\geq\frac{1}{2\eta_k}(\|y_{k+1}-y^*(x_k)\|^2-\|y_k-y^*(x_k)\|^2-\eta_k^2\|\tilde{y}_{k+1}-y_k\|^2)
	\end{align}
	Then, we have 
	\begin{align}
		0\geq&-\frac{L_g}{2}\|\tilde{y}_{k+1}-y_k\|^2+\frac{\mu_g}{2}\|y^*(x_k)-y_k\|^2+\frac{1}{\tau c_u}\| \tilde{y}_{k+1}-y_k \|^2\nonumber\\
		&-\frac{2}{\mu_g}\|\nabla_yg(x_k,y_k)-v_{k}\|^2-\frac{\mu_g}{4}\|y^*(x_k)-y_{k}\|^2-\frac{\mu_g}{4}\|y_k-\tilde{y}_{k+1}\|^2\nonumber\\
		&+\frac{1}{2\eta_k\tau c_u}(\|y_{k+1}-y^*(x_k)\|^2-\|y_k-y^*(x_k)\|^2-\eta_k^2\|\tilde{y}_{k+1}-y_k\|^2)
	\end{align}
	Hence we have 
	\begin{align}
		&\|y_{k+1}-y^*(x_k)\|^2\nonumber\\
  \leq&2\eta_k\tau c_u(\frac{L_g}{2}-\frac{1}{\tau c_u}+\frac{\mu_g}{4}+\frac{\eta_k}{2\tau c_u})\|\tilde{y}_{k+1}-y_k\|^2+(1-\frac{\mu_g\eta_k\tau c_u}{2})\|y^*(x_k)-y_k\|^2+\frac{4\eta_k\tau c_u}{\mu_g}\|\nabla_yg(x_k,y_k)-v_{k}\|^2\nonumber\\
		\leq&(1-\frac{\mu_g\eta_k\tau c_u}{2})\|y^*(x_k)-y_k\|^2+\frac{4\eta_k\tau c_u}{\mu_g}\|\nabla_yg(x_k,y_k)-v_{k}\|^2-2\eta_k\tau c_u(\frac{1}{2\tau c_u}-\frac{3L_g}{4})\|\tilde{y}_{k+1}-y_k\|^2\nonumber\\
        \leq&(1-\frac{\mu_g\eta_k\tau c_u}{2})\|y^*(x_k)-y_k\|^2+\frac{4\eta_k\tau c_u}{\mu_g}\|\nabla_yg(x_k,y_k)-v_{k}\|^2-\frac{3\eta_k}{4}\|\tilde{y}_{k+1}-y_k\|^2
	\end{align}
	using $\eta_k\leq1$, $\mu_g\leq L_g$ and $\tau\leq\frac{1}{6L_gc_u}$.
	
	Then, we have 
	\begin{align}
		&\|y_{k+1}-y^*(x_{k+1})\|^2\nonumber\\
		=&\|y_{k+1}-y^*(x_{k})+y^*(x_{k})-y^*(x_{k+1})\|^2\nonumber\\
		\leq&(1+\frac{\eta_k\tau\mu_g c_u}{4})\|y_{k+1}-y^*(x_{k})\|^2+(1+\frac{4}{\eta_k\tau\mu_g c_u})\|y^*(x_{k})-y^*(x_{k+1})\|^2\nonumber\\
		\leq&(1+\frac{\eta_k\tau\mu_gc_u}{4})\|y_{k+1}-y^*(x_{k})\|^2+(1+\frac{4}{\eta_k\tau\mu_g c_u})L_y^2\|x_{k}-x_{k+1}\|^2\nonumber\\
		\leq&(1+\frac{\eta_k\tau\mu_g c_u}{4})(1-\frac{\mu_g\eta_k\tau c_u}{2})\|y^*(x_k)-y_k\|^2+(1+\frac{\eta_k\tau\mu_g c_u}{4})\frac{4\eta_k\tau c_u}{\mu_g}\|\nabla_yg(x_k,y_k)-v_{k}\|^2\nonumber\\
		&-(1+\frac{\eta_k\tau\mu_gc_u}{4})\frac{3\eta_k}{4}\|\tilde{y}_{k+1}-y_k\|^2+(1+\frac{4}{\eta_k\tau\mu_g c_u})L_y^2\|x_{k}-x_{k+1}\|^2
	\end{align}
	Since $\eta_k\leq1$, $\mu_g\leq L_g$ and $\tau\leq\frac{1}{6L_gc_u}$, we have $\tau\leq\frac{1}{6L_gc_u}\leq\frac{1}{6\mu_gc_u}$ and $\eta_k\leq1\leq\frac{1}{6\tau L_gc_u}$. Then, we can obtain
    \begin{align}
        (1+\frac{\eta_k\tau\mu_g c_u}{4})(1-\frac{\mu_g\eta_k\tau c_u}{2})=&1-\frac{\mu_g\eta_k\tau c_u}{2}+\frac{\eta_k\tau\mu_g c_u}{4}-\frac{\eta_k^2\tau^2\mu_g^2 c_u^2}{8}\leq 1-\frac{\eta_k\tau\mu_g c_u}{4}\\
        -(1+\frac{\eta_k\tau\mu_gc_u}{4})\frac{3\eta_k}{4}\leq& -\frac{3\eta_k}{4}\\
        (1+\frac{\eta_k\tau\mu_g c_u}{4})\frac{4\eta_k\tau c_u}{\mu_g}\leq&\frac{25\eta_k\tau c_u}{6\mu_g}\\
        (1+\frac{4}{\eta_k\tau\mu_g c_u})L_y^2\leq&\frac{25L_y^2}{6\eta_k\tau\mu_g c_u}
    \end{align}
    Finally, we can obtain
    \begin{align}
		&\|y_{k+1}-y^*(x_{k+1})\|^2\nonumber\\
		\leq&(1-\frac{\eta_k\tau\mu_g c_u}{4})\|y^*(x_k)-y_k\|^2+\frac{25\eta_k\tau c_u}{6\mu_g}\|\nabla_yg(x_k,y_k)-v_{k}\|^2\nonumber\\
		&-\frac{3\eta_k}{4}\|\tilde{y}_{k+1}-y_k\|^2+\frac{25L_y^2}{6\eta_k\tau\mu_g c_u}\|x_{k}-x_{k+1}\|^2\nonumber\\
        \leq&(1-\frac{\eta_k\tau\mu_g c_u}{4})\|y^*(x_k)-y_k\|^2+\frac{25\eta_k\tau c_u}{6\mu_g}\|\nabla_yg(x_k,y_k)-v_{k}\|^2\nonumber\\
		&-\frac{3\eta_k}{4}\|\tilde{y}_{k+1}-y_k\|^2+\frac{25L_y^2\eta_k}{6\tau\mu_g c_u}\|x_{k}-\tilde{x}_{k+1}\|^2
	\end{align}
\end{proof}


\begin{lemma}\label{lemma:bound_of_z_in_ASZOGD}
	(\textbf{Descent in the gradient estimation error.\cite{huang2021biadam}}) Under Assumptions \ref{assump:upper_level}, \ref{assump:lower_level},  and Lemma \ref{lemma:bounds_of_hypergrad}, if $\alpha\in(0,1)$ and $\beta\in(0,1)$, we have
	\begin{align}
        &\mathbb{E}[\|\nabla f_{\delta}(x_{k+1},y_{k+1})+R_{k+1}-w_{k+1}\|^2]\nonumber\\
        \leq&(1-\alpha)\mathbb{E}[\|\nabla f_{\delta}(x_{k},y_{k})+R_{k}-w_k\|^2]+\alpha^2\sigma_f(d_2)+\frac{3}{\alpha}(\|R_k\|^2+\|R_{k+1}\|^2)+\frac{3}{\alpha}L_0^2\eta_k^2(\|{x}_k-\tilde{x}_{k+1}\|^2+\|{y}_k-\tilde{y}_{k+1}\|^2)\\
        &\mathbb{E}[\|\nabla g(x_{k+1},y_{k+1})-v_{k+1}\|^2]\nonumber\\
        \leq&(1-\beta)\mathbb{E}[\|\nabla g(x_{k},y_{k})-v_k\|^2]+\frac{2L_g^2}{\beta}\eta_k^2(\|{x}_k-\tilde{x}_{k+1}\|^2+\|{y}_k-\tilde{y}_{k+1}\|^2)
    \end{align}
     where $L_0=\max(L_1(\frac{d_2}{\delta}),L_2(\frac{d_2}{\delta}))$.
\end{lemma}


\subsection{Proof of the Convergence Rate in Theorem \ref{theorem:convergence_rate}}\label{section:proof_of_convergence_rate}
Here we first give a detailed version of Theorem \ref{theorem:convergence_rate} and then present the proof.
\begin{theorem}
Under Assumptions \ref{assump:upper_level}, \ref{assump:lower_level} and Lemma \ref{lemma:bound_of_nablaP},
with $\frac{1}{\mu_g}(1-\frac{1}{4(2\pi)^{1/4}\sqrt{d_2}L_p})\leq\eta<\frac{1}{\mu_g}$, $Q=\frac{1}{\mu_g\eta}\ln\frac{C_{gxy}C_{fy}K}{\mu_g}$, $0\leq a\leq 2$, $0<\gamma\leq\min\left\{\frac{1}{L_0^{2-a}},\frac{1}{2L_{F_{\delta}}(\frac{d_2}{\delta}) c_l\eta_k},\frac{1}{4c_l\left(\frac{125L_0^aL_y^2 c_l}{3\tau^2\mu_g^2 c_u^2}+(\frac{2}{3}L_0^2+\frac{6\mu_g^2L_g^2}{125L_0^2}) c_l\right)},\frac{2m^{1/2}}{9t}\right\}$, $0<\tau\leq\min\left\{\frac{1}{6L_gc_u},\frac{15L_0^a}{2\mu_gc_u\left(\frac{2}{3}L_0^2+\frac{6\mu_g^2L_g^2}{125L_0^2}\right)}\right\}$, $m\geq \max\{t^2,c_1^2t^2,c_2^2t^2\}$, $\alpha=c_1\eta_k$, $\beta=c_2\eta_k$, $\frac{9}{2}\gamma\leq c_1\leq\frac{m^{1/2}}{t}$, $\frac{125L_0^2}{3\mu_g^2}\leq c_2\leq\frac{m^{1/2}}{t}$, $L_0=\max(L_1(\frac{d_2}{\delta}),L_2(\frac{d_2}{\delta}))>1$, $\Phi_1
        =\mathbb{E}[F_{\delta}(x_{1})+\frac{10L_0^a c_l}{\tau\mu_g c_u}\|y_{1}-y^*(x_{1})\|^2 + c_l(\|w_{1}-\bar{\nabla} f_{\delta}(x_{1},y_{1})-R_{1}\|^2+\|\nabla_y g(x_{1},y_{1})-v_{1}\|^2)]$, and $\eta_k=\frac{t}{(m+k)^{1/2}}$, $t>0$, we have
\begin{align}
        &\frac{1}{K}\sum_{k=1}^K\mathbb{E}[\frac{1}{2}\|\nabla F_{\delta}(x_k)\|]
        \leq\frac{2m^{1/4}\sqrt{G}}{\sqrt{Kt}}+\frac{2\sqrt{G}}{(Kt)^{1/4}}.
    \end{align}
    where $G=\frac{\Phi_1-\Phi^*}{\gamma c_l}+\frac{17t}{4K^2}(m+K)^{1/2}+\frac{4}{3tK^2}(m+K)^{3/2}+(m\sigma_f(d_2))t^2\ln (m+K)$.
\end{theorem}
\begin{proof}
    Setting  $\eta_k=\frac{t}{(m+k)^{1/2}}$ and $m\geq t^2$, we have $\eta_k\leq 1$. Due to $m\geq (c_1t)^2$, we have $\alpha=c_1\eta_k\leq\frac{c_1t}{m^{1/2}}\leq 1$ . Due to $m\geq(c_2t)^2$, we have $\beta=c_2\eta_k\leq\frac{c_2t}{m^{1/2}}\leq 1$. Also, we have $c_1,c_2\leq \frac{m^{1/2}}{t}$. Then using the above lemmas, we have

    \begin{align}
        &\mathbb{E}[\|\nabla f_{\delta}(x_{k+1},y_{k+1})+R_{k+1}-w_{k+1}\|^2]-\mathbb{E}[\|\nabla f_{\delta}(x_{k},y_{k})+R_{k}-w_k\|^2]\nonumber\\
        \leq&-\alpha\mathbb{E}[\|\nabla f_{\delta}(x_{k},y_{k})+R_{k}-w_k\|^2]+\alpha^2\sigma_f(d_2)+\frac{3}{\alpha}(\|R_k\|^2+\|R_{k+1}\|^2)\nonumber\\
        &+\frac{3}{\alpha}L_0^2\eta_k^2(\|{x}_k-\tilde{x}_{k+1}\|^2+\|{y}_k-\tilde{y}_{k+1}\|^2)\nonumber\\
        \leq&-c_1\eta_k\mathbb{E}[\|\nabla f_{\delta}(x_{k},y_{k})+R_{k}-w_k\|^2]+c_1^2\eta_k^2\sigma_f(d_2)+\frac{3}{c_1\eta_k}(\|R_k\|^2+\|R_{k+1}\|^2)\nonumber\\
        &+\frac{3}{c_1}L_0^2\eta_k(\|{x}_k-\tilde{x}_{k+1}\|^2+\|{y}_k-\tilde{y}_{k+1}\|^2)\nonumber\\
        \leq&-\frac{9}{2}\gamma\eta_k\mathbb{E}[\|\nabla f_{\delta}(x_{k},y_{k})+R_{k}-w_k\|^2]+\frac{m}{t^2}\eta_k^2\sigma_f(d_2)+\frac{2}{3\eta_k}(\|R_k\|^2+\|R_{k+1}\|^2)\nonumber\\
        &+\frac{2}{3}L_0^2\eta_k(\|{x}_k-\tilde{x}_{k+1}\|^2+\|{y}_k-\tilde{y}_{k+1}\|^2)
    \end{align}
    where the last inequality holds by $\frac{9}{2}\gamma\leq c_1\leq\frac{m^{1/2}}{t}$.
    \begin{align}
       &\mathbb{E}[\|\nabla g(x_{k+1},y_{k+1})-v_{k+1}\|^2]-\mathbb{E}[\|\nabla g(x_{k},y_{k})-w_k\|^2]\nonumber\\
        \leq&-\beta\mathbb{E}[\|\nabla g(x_{k},y_{k})-v_k\|^2]+\frac{2L_g^2}{\beta}\eta_k^2(\|{x}_k-\tilde{x}_{k+1}\|^2+\|{y}_k-\tilde{y}_{k+1}\|^2)\nonumber\\
        \leq&-c_2\eta_k\mathbb{E}[\|\nabla g(x_{k},y_{k})-v_k\|^2]+\frac{2L_g^2}{c_2}\eta_k(\|\tilde{x}_k-x_{k+1}\|^2+\|\tilde{y}_k-y_{k+1}\|^2)\nonumber\\
        \leq&-\frac{125L_0^2}{3\mu_g^2}\eta_k\mathbb{E}[\|\nabla g(x_{k},y_{k})-v_k\|^2]+\frac{6\mu_g^2L_g^2}{125L_0^2}\eta_k(\|{x}_k-\tilde{x}_{k+1}\|^2+\|{y}_k-\tilde{y}_{k+1}\|^2)
    \end{align}
    where the last inequality hold by $\frac{125L_0^2}{3\mu_g^2}\leq c_2\leq\frac{m^{1/2}}{t}$.

    In addition, we have 
    \begin{align}
        &F_{\delta}(x_{k+1})-F_{\delta}(x_{k})\nonumber\\
        \leq &\eta_k\gamma c_l\left(2L_1^2(\frac{d_2}{\delta})\|y^*(x_k)-y_k\|^2+2\|\nabla f_{\delta}(x_k,y_k)-w_{k}\|^2\right)-\frac{\eta_k}{2\gamma c_l}\|\tilde{x}_{k+1}-x_k\|^2\nonumber\\
        \leq& 2 \eta_k\gamma c_l L_0^2\|y^*(x_k)-y_k\|^2+2\eta_k\gamma c_l\|\nabla f_{\delta}(x_k,y_k)-w_{k}\|^2-\frac{\eta_k}{2\gamma c_l}\|\tilde{x}_{k+1}-x_k\|^2\nonumber\\
        \leq& 2 \eta_k\gamma c_l L_0^2\|y^*(x_k)-y_k\|^2+4\eta_k\gamma c_l\|\nabla f_{\delta}(x_k,y_k)-w_{k}-R_k\|^2+4\eta_k\gamma c_l\|R_k\|^2-\frac{\eta_k}{2\gamma c_l}\|\tilde{x}_{k+1}-x_k\|^2
    \end{align}
    
    We can also have 
    \begin{align}
	&\|y_{k+1}-y^*(x_{k+1})\|^2-\|y^*(x_k)-y_k\|^2\nonumber\\
        \leq&-\frac{\eta_k\tau\mu_g c_u}{4}\|y^*(x_k)-y_k\|^2+\frac{25\eta_k\tau c_u}{6\mu_g}\|\nabla_yg(x_k,y_k)-v_{k}\|^2\nonumber\\
		&-\frac{3\eta_k}{4}\|\tilde{y}_{k+1}-y_k\|^2+\frac{25L_y^2\eta_k}{6\tau\mu_g c_u}\|x_{k}-\tilde{x}_{k+1}\|^2
    \end{align}

    Then, we define a Lyapunov function, for any $k\geq 1$,
     \begin{align}
        &\Phi_{k+1}\nonumber\\
        =&\mathbb{E}[F_{\delta}(x_{k+1})+\frac{10L_0^a c_l}{\tau\mu_g c_u}\|y_{k+1}-y^*(x_{k+1})\|^2 + c_l(\|w_{k+1}-\bar{\nabla} f_{\delta}(x_{k+1},y_{k+1})-R_{k+1}\|^2\nonumber\\
        &+\|\nabla_y g(x_{k+1},y_{k+1})-v_{k+1}\|^2)]
    \end{align}
    We have 
    \begin{align}
        &\Phi_{k+1}-\Phi_{k}\nonumber\\
        =&\mathbb{E}[F_{\delta}(x_{k+1})-F_{\delta}(x_{k})]+\frac{10L_0^a c_l}{\tau\mu c_u}\mathbb{E}[\|y_{k+1}-y^*(x_{k+1})\|^2- \|y_{k}-y^*(x_{k})\|^2]\nonumber\\
        &+ c_l\mathbb{E}[\|w_{k+1}-\bar{\nabla} f(x_{k+1},y_{k+1})-R_{k+1}\|^2-\|w_k-\bar{\nabla} f(x_k,y_k)-R_k\|^2]\nonumber\\
        &+ c_l\mathbb{E}[\|\nabla_y g(x_{k+1},y_{k+1})-v_{k+1}\|^2-\|\nabla_y g(x_{k},y_{k})-v_{k}\|^2]\nonumber\\
        \leq&2 \eta_k\gamma c_l L_0^2\|y^*(x_k)-y_k\|^2+4\eta_k\gamma c_l\|\nabla f_{\delta}(x_k,y_k)-w_{k}-R_k\|^2+4\eta_k\gamma c_l\|R_k\|^2-\frac{\eta_k}{2\gamma c_l}\|\tilde{x}_{k+1}-x_k\|^2\nonumber\\
        &+\frac{10L_0^ac_l}{\tau\mu_g c_u}(-\frac{\eta_k\tau\mu_g c_u}{4}\|y^*(x_k)-y_k\|^2+\frac{25\eta_k\tau c_u}{6\mu_g}\|\nabla_yg(x_k,y_k)-v_{k}\|^2-\frac{3\eta_k}{4}\|\tilde{y}_{k+1}-y_k\|^2+\frac{25L_y^2\eta_k}{6\tau\mu_g c_u}\|x_{k}-\tilde{x}_{k+1}\|^2)\nonumber\\
        &+ c_l(-\frac{9}{2}\gamma\eta_k\mathbb{E}[\|\nabla f_{\delta}(x_{k},y_{k})+R_{k}-w_k\|^2]+\frac{m}{t^2}\eta_k^2\sigma_f(d_2)+\frac{2}{3\eta_k}(\|R_k\|^2+\|R_{k+1}\|^2)\nonumber\\
        &+\frac{2}{3}L_0^2\eta_k(\|{x}_k-\tilde{x}_{k+1}\|^2+\|{y}_k-\tilde{y}_{k+1}\|^2))\nonumber\\
        &+ c_l(-\frac{125L_0^2}{3\mu_g^2}\eta_k\mathbb{E}[\|\nabla g(x_{k},y_{k})-v_k\|^2]+\frac{6\mu_g^2L_g^2}{125L_0^2}\eta_k(\|{x}_k-\tilde{x}_{k+1}\|^2+\|{y}_k-\tilde{y}_{k+1}\|^2))\nonumber\\
        \leq&(2 \eta_k\gamma c_l L_0^2-\frac{5L_0^a c_l\eta_k}{2})\|y^*(x_k)-y_k\|^2+(4\eta_k\gamma c_l -\frac{9 c_l\gamma}{2}\eta_k)\|\nabla f_{\delta}(x_k,y_k)-w_{k}-R_k\|^2\nonumber\\
        &+(\frac{125L_0^a c_l\eta_k}{3\mu_g^2 }-c_l\frac{125L_0^2}{3\mu_g^2}\eta_k)\|\nabla_yg(x_k,y_k)-v_{k}\|^2\nonumber\\
        &+\left(-\frac{\eta_k}{2\gamma c_l}+\frac{125L_0^aL_y^2 c_l\eta_k}{3\tau^2\mu_g^2 c_u^2}+(\frac{2}{3}L_0^2+\frac{6\mu_g^2L_g^2}{125L_0^2})\eta_k c_l\right)\|\tilde{x}_{k+1}-x_k\|^2\nonumber\\
        &+\left(-\frac{15L_0^ac_l\eta_k}{2\tau\mu_g c_u}+(\frac{2}{3}L_0^2+\frac{6\mu_g^2L_g^2}{125L_0^2})\eta_k c_l\right)\|\tilde{y}_{k+1}-y_k\|^2\nonumber\\
        &+4\eta_k\gamma c_l\|R_k\|^2+\frac{2\gamma c_l}{3\eta_k}(\|R_k\|^2+\|R_{k+1}\|^2)+\frac{m}{t^2}\gamma c_l\eta_k^2\sigma_f(d_2)\nonumber\\
        \leq&-\frac{L_0^2\gamma c_l\eta_k}{2}\|y^*(x_k)-y_k\|^2 -\frac{\gamma c_l}{2}\eta_k\|\nabla f_{\delta}(x_k,y_k)-w_{k}-R_k\|^2\nonumber\\
        &-\frac{\eta_k}{4\gamma c_l}\|\tilde{x}_{k+1}-x_k\|^2+4\eta_k\gamma c_l\|R_k\|^2+\frac{2\gamma c_l}{3\eta_k}(\|R_k\|^2+\|R_{k+1}\|^2)\nonumber\\
        &+\frac{m}{t^2}\gamma c_l\eta_k^2\sigma_f(d_2)\nonumber\\
    \end{align}
    where the last inequality is due to $0\leq a\leq 2$,$\gamma\leq\min\{\frac{1}{L_0^{2-a}},\frac{1}{4c_l\left(\frac{125L_0^aL_y^2 c_l}{3\tau^2\mu_g^2 c_u^2}+(\frac{2}{3}L_0^2+\frac{6\mu_g^2L_g^2}{125L_0^2}) c_l\right)},\frac{2m^{1/2}}{9t}\}$, $0<\tau\leq\frac{15L_0^a}{2\mu_gc_u\left(\frac{2}{3}L_0^2+\frac{6\mu_g^2L_g^2}{125L_0^2}\right)}$ and $L_0>1$.

    Then, rearranging the above inequality, we have
    \begin{align}
        &\frac{\gamma c_l\eta_k}{4}\left(2L_0^2\|y^*(x_k)-y_k\|^2 +2\|\nabla f_{\delta}(x_k,y_k)-w_{k}-R_k\|^2+\|R_k\|^2+\frac{1}{\gamma^2 c_l^2}\|\tilde{x}_{k+1}-x_k\|^2\right)\nonumber\\
        \leq&\frac{17}{4}\eta_k\gamma c_l\|R_k\|^2+\frac{2\gamma c_l}{3\eta_k}(\|R_k\|^2+\|R_{k+1}\|^2)+\frac{m}{t^2}\gamma c_l\eta_k^2\sigma_f(d_2)+\Phi_k-\Phi_{k+1}
    \end{align}
    Taking the average over $k=1,\cdots,K$ on both sides and using $\eta_k\geq\eta_K$, $Q=\frac{1}{\mu_g\eta}\ln\frac{C_{gxy}C_{fy}K}{\mu_g}$, $\eta_k=\frac{t}{(m+k)^{1/2}}$ and $\Phi_1
        =\mathbb{E}[F_{\delta}(x_{1})+\frac{10L_0^a c_l}{\tau\mu_g c_u}\|y_{1}-y^*(x_{1})\|^2 + c_l(\|w_{1}-\bar{\nabla} f_{\delta}(x_{1},y_{k+1})-R_{1}\|^2+\|\nabla_y g(x_{1},y_{1})-v_{1}\|^2)]$, we have
    \begin{align}
        &\frac{1}{K}\sum_{k=1}^K\mathbb{E}[\frac{1}{4}\left(2L_0^2\|y^*(x_k)-y_k\|^2 +2\|\nabla f_{\delta}(x_k,y_k)-w_{k}-R_k\|^2+\|R_k\|^2+\frac{1}{\gamma^2 c_l^2}\|\tilde{x}_{k+1}-x_k\|^2\right)]\nonumber\\
        \leq&\frac{1}{K\eta_K}\left(\frac{\Phi_1-\Phi^*}{\gamma c_l}+\frac{17}{4K^2}\sum_{k=1}^K\eta_k+\frac{4}{3K^2}\sum_{k=1}^K\frac{1}{\eta_k}+(\frac{m}{t^2}\sigma_f(d_2))\sum_{k=1}^K\eta_k^2\right)\nonumber\\
        \leq&\frac{(m+K)^{1/2}}{Kt}\left(\frac{\Phi_1-\Phi^*}{\gamma c_l}+\frac{17t}{4K^2}(m+K)^{1/2}+\frac{4}{3tK^2}(m+K)^{3/2}+(m\sigma_f(d_2))t^2\ln (m+K)\right)
    \end{align}
    According to the Jesen's inequality, we have
    \begin{align}
        &\frac{1}{K}\sum_{k=1}^K\mathbb{E}[\frac{1}{2}\left(\sqrt{2}L_0^2\|y^*(x_k)-y_k\| +\sqrt{2}\|\nabla f_{\delta}(x_k,y_k)-w_{k}-R_k\|+\|R_k\|+\frac{1}{\gamma c_l}\|\tilde{x}_{k+1}-x_k\|\right)]\nonumber\\
        \leq&\left(\frac{4}{K}\sum_{k=1}^K\frac{1}{4}\left(2L_0^2\|y^*(x_k)-y_k\|^2 +2\|\nabla f_{\delta}(x_k,y_k)-w_{k}-R_k\|^2+\|R_k\|^2+\frac{1}{\gamma^2 c_l^2}\|\tilde{x}_{k+1}-x_k\|^2\right)\right)^{1/2}\nonumber\\
        \leq&\frac{2(m+K)^{1/4}}{\sqrt{Kt}}\sqrt{\frac{\Phi_1-\Phi^*}{\gamma c_l}+\frac{17t}{4K^2}(m+K)^{1/2}+\frac{4}{3tK^2}(m+K)^{3/2}+(m\sigma_f(d_2))t^2\ln (m+K)}\nonumber\\
        \leq&\frac{2m^{1/4}\sqrt{G}}{\sqrt{Kt}}+\frac{2\sqrt{G}}{(Kt)^{1/4}}
    \end{align}
    where $G=\frac{\Phi_1-\Phi^*}{\gamma c_l}+\frac{17t}{4K^2}(m+K)^{1/2}+\frac{4}{3tK^2}(m+K)^{3/2}+(m\sigma_f(d_2))t^2\ln (m+K)$.
    
    Let $\mathcal{G}\left(x_k,\nabla F_{\delta}(x_k),\hat{\gamma}\right)=\frac{1}{\hat{\gamma}}\left( x_k-\mathcal{P}_{\mathcal{X}}\left(x_k-\hat{\gamma}\nabla F_{\delta}(x_k)\right)\right)$
 
    \begin{align}
	&\|\nabla F_{\delta}(x_k)\| \nonumber\\  
     =&\|\mathcal{G}\left(x_k,\nabla F_{\delta}(x_k),\hat{\gamma}\right)\|\nonumber\\
	    \leq&\|\mathcal{G}\left(x_k,\nabla F_{\delta}(x_k),\hat{\gamma}\right)-\mathcal{G}(x_t,w_k,\hat{\gamma})\|+\|\mathcal{G}(x_t,w_k,\hat{\gamma})\|\nonumber\\
	    \leq&\|\nabla F_{\delta}(x_k)-w_k\|+\|\mathcal{G}(x_t,w_k,\hat{\gamma})\|\nonumber\\\
	   \leq&\|w_k-{\nabla} f(x_k,y_k)-R_k\|+\|R_k\|^2+L_0\|y^*(x_k)-y_k\|+\frac{1}{\gamma c_l}\|x_k-\tilde{x}_{k+1}\|.
	\end{align}

    Finally, we can obtain 
    \begin{align}
        &\frac{1}{K}\sum_{k=1}^K\mathbb{E}[\frac{1}{2}\|\nabla F_{\delta}(x_k)\|]
        \leq\frac{2m^{1/4}\sqrt{G}}{\sqrt{Kt}}+\frac{2\sqrt{G}}{(Kt)^{1/4}}.
    \end{align}

    Since the random count $r\in\{1,\cdot, K\}$ is uniformly sampled, we have 

    \begin{align}
        &\mathbb{E}[\frac{1}{2}\|\nabla F_{\delta}(x_r)\|]=\frac{1}{K}\sum_{k=1}^K\mathbb{E}[\frac{1}{2}\|\nabla F_{\delta}(x_k)\|]
        \leq\frac{2m^{1/4}\sqrt{G}}{\sqrt{Kt}}+\frac{2\sqrt{G}}{(Kt)^{1/4}}.
    \end{align}

    By Proposition \ref{prop2}, we have $\nabla F_{\delta}(x_r)\in\bar{\partial}_{\delta}F(x_r)$. This implies that 
    \begin{align}
        &\min\{\|h\|:h\in\bar{\partial}_{\delta} F(x_r)\}\leq \mathbb{E}[\|\nabla F_{\delta}(x_r)\|]
        \leq\frac{4m^{1/4}\sqrt{G}}{\sqrt{Kt}}+\frac{4\sqrt{G}}{(Kt)^{1/4}}.
    \end{align}
    That completes the proof.
    
\end{proof}

\end{document}